\documentclass[11pt]{article}
\usepackage{epsf,amsmath,amsfonts,amsthm,graphicx,color}
\usepackage{stmaryrd,comment}
\numberwithin{equation}{section}

\begin{document}

\theoremstyle{plain}
\newtheorem{Lemma}{Lemma}[section]
\newtheorem{Prop}[Lemma]{Proposition}
\newtheorem{Thm}[Lemma]{Theorem}
\newtheorem{Cor}[Lemma]{Corollary}

\theoremstyle{definition}
\newtheorem{Def}[Lemma]{Definition}
\newtheorem{Rk}[Lemma]{Remark}
\newtheorem{Example}[Lemma]{Example}
\newtheorem{Exercise}[Lemma]{Exercise}

\newcommand{\Natural}{\mbox{${\bf N}$}}
\newcommand{\Integer}{\mbox{${\bf Z}$}}
\newcommand{\Real}{\mbox{${\bf R}$}}
\newcommand{\Complex}{\mbox{${\bf C}$}}

\newcommand{\Eps}{\varepsilon}

\newcommand{\Sfrac}[2]{\mbox{\small$\frac{#1}{#2}$\normalsize}}
\newcommand{\Half}{\Sfrac{1}{2}}
\newcommand{\Vect}[1]{{\bf #1}}
\newcommand{\Grad}[1]{\nabla #1}
\newcommand{\Gradp}[1]{\nabla' #1}
\newcommand{\Gradx}[1]{\nabla_x #1}
\newcommand{\Gradxp}[1]{\nabla_{x'} #1}
\newcommand{\GradAlphap}[1]{\nabla_{\alpha}' #1}
\newcommand{\Div}[1]{\text{div}\left[#1\right]}
\newcommand{\Divp}[1]{\text{div}'\left[#1\right]}
\newcommand{\Divx}[1]{\text{div}_x \left[#1\right]}
\newcommand{\Divxp}[1]{\text{div}_{x'} \left[#1\right]}
\newcommand{\Curl}[1]{\nabla \times #1}
\newcommand{\CurlOp}[1]{\nabla \times \left[#1\right]}
\newcommand{\Laplacian}[1]{\Delta #1}
\newcommand{\Laplacianp}[1]{\Delta' #1}
\newcommand{\Laplacianx}[1]{\Delta_x #1}
\newcommand{\Laplacianxp}[1]{\Delta_{x'} #1}
\newcommand{\Biharmonic}[1]{\Delta^2 #1}
\newcommand{\FT}[1]{{\cal F} \left\{ #1 \right\}}
\newcommand{\FTI}[1]{{\cal F}^{-1} \left\{ #1 \right\}}
\newcommand{\Variation}[2]{\delta_{#2} #1}

\newcommand{\Norm}[2]{\left\|#1\right\|_{#2}}
\newcommand{\LeftNorm}[1]{\left\|#1\right.}
\newcommand{\RightNorm}[2]{\left.#1\right\|_{#2}}
\newcommand{\SupNorm}[1]{\left|#1\right|_{L^{\infty}}}
\newcommand{\HolderNorm}[2]{\left|#1\right|_{C^{#2}}}
\newcommand{\SobNorm}[2]{\left\|#1\right\|_{H^{#2}}}

\newcommand{\InnerProd}[2]{\left\langle#1,#2\right\rangle}
\newcommand{\DotProd}[2]{\left\langle#1,#2\right\rangle}
\newcommand{\Abs}[1]{\left|#1\right|}
\newcommand{\Mod}[1]{\left|#1\right|}
\newcommand{\Angle}[1]{\langle #1 \rangle}
\newcommand{\RealPart}[1]{\text{Re\{}#1\text{\}}}
\newcommand{\ImagPart}[1]{\text{Im\{}#1\text{\}}}
\newcommand{\Null}[1]{\mbox{${\cal N}$}(#1)}
\newcommand{\Ran}[1]{\text{ran}(#1)}
\newcommand{\Ker}[1]{\text{ker}(#1)}
\newcommand{\Dim}[1]{\text{dim}(#1)}
\newcommand{\Rank}[1]{\text{rank}(#1)}
\newcommand{\Det}[1]{\mbox{det} #1}
\newcommand{\Span}[1]{\text{span}(#1)}
\newcommand{\sgn}{\text{sgn}}

\newcommand{\sech}{\mbox{$\mathrm{sech}$}}
\newcommand{\csch}{\mbox{$\mathrm{csch}$}}

\newcommand{\Intersect}{\cap}
\newcommand{\Union}{\cup}

\newcommand{\dftl}[1]{\; d#1}
\newcommand{\dV}{\dftl{V}}
\newcommand{\dS}{\dftl{S}}
\newcommand{\dx}{\dftl{x}}
\newcommand{\dy}{\dftl{y}}
\newcommand{\dz}{\dftl{z}}
\newcommand{\ds}{\dftl{s}}
\newcommand{\dt}{\dftl{t}}
\newcommand{\du}{\dftl{u}}
\newcommand{\dsigma}{\dftl{\sigma}}

\newcommand{\BigOh}[1]{\mathcal{O}(#1)}
\newcommand{\LittleOh}[1]{\mathcal{o}(#1)}

\newcommand{\px}{\partial_x}
\newcommand{\py}{\partial_y}
\newcommand{\pz}{\partial_z}
\newcommand{\pt}{\partial_t}

\newcommand{\sumn}{\sum_{n=0}^{\infty}}
\newcommand{\sumno}{\sum_{n=1}^{\infty}}
\newcommand{\sumk}{\sum_{k=-\infty}^{\infty}}
\newcommand{\sump}{\sum_{p=-\infty}^{\infty}}
\newcommand{\sumq}{\sum_{q=-\infty}^{\infty}}
\newcommand{\sumr}{\sum_{r=0}^{\infty}}

\newcommand{\summ}{\sum_{m=0}^{\infty}}

\newcommand{\be}{\begin{equation}}
    \newcommand{\ee}{\end{equation}}
\newcommand{\bes}{\begin{equation*}}
    \newcommand{\ees}{\end{equation*}}
\newcommand{\bse}{\begin{subequations}}
    \newcommand{\ese}{\end{subequations}}

\newcommand{\Schrodinger}{Schr\"odinger}
\newcommand{\Holder}{H\"older}
\newcommand{\Calderon}{Calder\'{o}n}
\newcommand{\Pade}{Pad\'{e}}

\newcommand{\Question}[1]{\fbox{ {\bf Q: #1} }}
\newcommand{\Corrected}[1]
{\noindent \rule{\linewidth}{.75mm} \\ {\bf CORRECTED UP TO HERE (#1)} \\ \rule{\linewidth}{.75mm}}
\newcommand{\void}[1]{}

\newcommand{\RevOne}[1]{\textcolor{red}{#1}}
\newcommand{\RevTwo}[1]{\textcolor{blue}{#1}}
\newcommand{\RevThree}[1]{\textcolor{green}{#1}}

\newcommand{\cL}{\mathcal{L}}
\newcommand{\cP}{\mathcal{P}}

\newcommand{\ku}{k^{(u)}}
\newcommand{\kv}{k^{(v)}}
\newcommand{\kw}{k^{(w)}}
\newcommand{\km}{k^{(m)}}
\newcommand{\kz}{k_0}
\newcommand{\gammau}{\gamma^{(u)}}
\newcommand{\gammav}{\gamma^{(v)}}
\newcommand{\gammaw}{\gamma^{(w)}}
\newcommand{\gammam}{\gamma^{(m)}}
\newcommand{\epsu}{\epsilon^{(u)}}
\newcommand{\epsv}{\epsilon^{(v)}}
\newcommand{\epsw}{\epsilon^{(w)}}
\newcommand{\epsm}{\epsilon^{(m)}}
\newcommand{\epsz}{\epsilon_0}
\newcommand{\muu}{\mu^{(u)}}
\newcommand{\muv}{\mu^{(v)}}
\newcommand{\muw}{\mu^{(w)}}
\newcommand{\mum}{\mu^{(m)}}
\newcommand{\muz}{\mu_0}
\newcommand{\chiu}{\chi^{(u)}}
\newcommand{\chiv}{\chi^{(v)}}
\newcommand{\chiw}{\chi^{(w)}}
\newcommand{\chim}{\chi^{(m)}}
\newcommand{\bchi}{\bar{\chi}}
\newcommand{\etaz}{\eta_0}
\newcommand{\sigmaz}{\sigma_0}
\newcommand{\Jz}{J_0}
\newcommand{\Vz}{V_0}
\newcommand{\bgammaell}{\bar{\gamma}_{\ell}}
\newcommand{\bgammar}{\bar{\gamma}_r}
\newcommand{\uk}{\underline{k}}

\newcommand{\tM}{\tilde{M}}

\newcommand{\tdelta}{\tilde{\delta}}

\newcommand{\Jump}[1]{\left\llbracket #1 \right\rrbracket}

\newcommand{\AAEE}{A^E}
\newcommand{\AAHH}{A^H}

%
%

\title{A High--Order Perturbation of Envelopes (HOPE) Method for 
  Electromagnetic Scattering by Periodic Inhomogeneous Chiral Media}
\author{
David P.\ Nicholls, \\
Department of Mathematics, Statistics, and Computer Science, \\
University of Illinois at Chicago, \\
Chicago, IL 60607
\and
Liet Vo, \\
School of Mathematical and Statistical Sciences, \\
The University of Texas Rio Grande Valley, \\
Edinburg, TX 78539
}

\maketitle

\begin{abstract}
Chirality plays an important role in many optical phenomena
and it is crucial to have efficient and accurate numerical
algorithms to simulate solutions in this setting. In this
paper we discuss a High--Order Spectral method coupled to
geometric regular perturbation theory, which results in just
such a method. We view the chirality of the laterally
periodic, constant permittivity/permeability structure
as deviating from a background value, and demonstrate
analyticity of the field scattered by a plane electromagnetic
wave with respect to this deformation. This
analyticity is not only with respect to chirality
deformations of arbitrarily large real size, but also
joint in the spatial variables. We also show how the
High--Order Perturbation of Envelopes (HOPE) recursions
which we used to establish these results can be implemented
as a rapid, robust, and reliable numerical algorithm.
\end{abstract}

%
%

\section{Introduction}
\label{Sec:Intro}

Chirality is a crucial property in chemistry \cite{ChiralityChemistry2024},
biology \cite{ChiralityBiology2019}, and physics \cite{ChiralityPhysics2023}
which occurs in nature
(e.g., hands, shells, amino acids, DNA/RNA)
and has been exploited by humans to great advantage
(e.g., gloves, spiral staircases, screws).
A chiral object is one that cannot be superposed onto its
reflection (enantiomer), with the classic example being
a hand (whose Greek name gives us the term ``chiral'') leading to the 
natural classification of such an object
by its ``handedness.'' Sometimes, enantiomers (an object and
its reflection) exhibit quite different
behaviors, with a most tragic example giving rise to the 
thalidomide disaster of the late 1950s and early 1960s
\cite{Tokunaga18}. Thalidomide has two non-superposable enantiomers:
(S) which is teratogenic and (R) which is therapeutic \cite{Blaschke79}.
In the 1950s the latter was administered to animals and the successful experiments
led to its rapid prescription
in humans, notably pregnant women for treatment of morning sickness.
However, in humans this enantiomer rapidly interconverts \textit{in vivo}
leading to disastrous consequences. The long--standing 
``thalidomide paradox'' is why this racemization of thalidomide
in humans but not the experimental trial animals?
Tokunaga \textit{et al} \cite{Tokunaga18} put forward a hypothesis, 
but, regardless of
the resolution, it is clearly of the highest importance to be able to
discriminate amongst enantiomers of chemical species. (Notably, the
reputation of thalidomide has been rehabilitated in recent years with
the discovery of its beneficial effects in treating leprosy and certain
types of cancers \cite{GroganWinston23}.)

For this reason we take up the crucially important question of the
sensitive and robust detection of chiral molecules \cite{Wang25},
though our developments can be applied to more general scenarios
involving the optical response of chiral materials; see 
Mun \textit{et al} \cite{Mun20} and Khaliq \textit{et al}
\cite{Khaliq23} for instance. Regarding chiral sensors, Wang
\textit{et al} \cite{Wang25} provide an exhaustive overview
of a number of different strategies with the common theme being
that great
advantage can be gained by incorporating chirality into the sensor
itself. Amongst the many schemes outlined, and of particular interest
to us, are the ones utilizing (surface) plasmonic responses. For instance, those
based upon chiral metasurfaces formed from chiral gold nanorods \cite{Li25},
chiral gold nanorods in a palladium nanostructure \cite{Luo25}, 
arrays of chiral gold helicoid crystals \cite{Kim22}, and
chiral nanoparticles exhibiting localized surface plasmon resonances
(LSPRs) \cite{Jeong16}. A particularly compelling application of a
Surface Plasmon Resonance (SPR)
sensor based upon chiral gold nanorods is given in Kumar \textit{et al}
\cite{Kumar18} who use it to detect Parkinson's disease.


While there are several approaches to building an SPR sensor
\cite{Homola08}, a convenient and widely used method is to
periodically pattern a structure at the nanoscale. This
precise patterning gives one path to overcoming the momentum mismatch
required to excite a surface plasmon \cite{Raether88,Maier07,EnochBonod12}.
For this reason we focus our attention upon the interaction of
linear electromagnetic waves with laterally periodic structures
which, of course, arise in many other areas in science and engineering
(e.g., underwater acoustics \cite{TaroudakisMakrakis01},
remote sensing \cite{TKS85}, and nondestructive testing 
\cite{S02}).

With the overwhelming technological importance of these periodically
patterned devices, it is not surprising to learn that the entire
host of classical numerical schemes has been brought to the task
of their simulation
\cite{GallinetButetMartin15}. More specifically, in the context
of electromagnetic scattering, one can consider the 
Finite Difference Time Domain (FDTD) \cite{TafloveHagness00} and 
Finite Difference Frequency Domain (FDFD) methods \cite{Rumpf22},
the Discontinuous Galerkin (DG) \cite{BuschKonigNiegmann11}
and Finite Element Methods (FEM) \cite{Jin02},
the Volume Integral Equation (VIE) algorithm \cite{MartinPiller98},
and the Discrete Dipole Approximation \cite{DraineFlatau94}.
More specifically, for electromagnetics applications are 
the Rigorous Coupled--Wave Analysis (RCWA) method
\cite{MoharamGaylord81,MoharamPommetGrannGaylord95,LalanneMorris96},
the Fourier Modal Method (FMM) \cite{KimParkLee12},
and Planewave Eigensolvers \cite{MRBJA93,JohnsonJoannopoulos01}.

As we shall see, the response of a chiral medium to illumination by
linear electromagnetic waves is not well--modeled by the classical
constitutive relations. Several alternative models have been put forth
and the Drude--Born--Federov (DBF) relations \eqref{Eqn:DBF:Constit} are
a popular choice \cite{LVV89Book,Lakhtakia94} that we pursue. The resulting
PDEs require specialized treatment and a number of authors have
made significant contributions to a rigorous weak formulation of
the problem and its subsequent FEM/BEM simulation \cite{Stratis99}.
We point out early papers on bounded chiral obstacles in two and
three dimensions
\cite{Rojas94,AthanasiadisMartinStratis99,AthanasiadisCostakisStratis00},
but the careful rigorous analysis for the periodic gratings we consider
began with the pioneering work of Ammari \& Bao 
\cite{AmmariBao98,AmmariBao03,AmmariBao08}. This was later pursued by 
Zhang and collaborators for bounded obstacles and periodic gratings
\cite{ZhangMa05,ZhangMa07,ZhangGuoGongWang12}. While not the topic of
our current investigations, there have been additional developments on
the inverse problem of detecting a chiral obstacle given far field
measurements (see \cite{deMonvelBoutetShepelsky97,Gerlach99,PotthastStratis03}
and \cite{Nguyen16,FengWangZhang21,GuoWang22}).

In this paper we take a slightly different approach to the numerical
simulation of periodic chiral structures, which follows the 
High--Order Perturbation of Envelopes (HOPE) methodology we developed
for solving the Helmholtz equation \cite{Nicholls19b} and 
Maxwell equations \cite{NichollsVo23,NichollsVo24} which arise in scattering
from achiral media. As in \cite{Nicholls19b}, we focus on $y$--invariant
structures illuminated transversely so that the solution of a scalar
Helmholtz equation delivers the vectorial scattered field;
see \S~\ref{Sec:2DChirality}. (We will return to the more general three--dimensional
case in future work.) This algorithm takes a perturbative point of view
by considering chirality functions, $\chi(x,z)$ (see \S~\ref{Sec:Govern}),
which are laterally periodic deviations from a trivial one, e.g.,
\bes
\chi(x,z) = \bchi - \delta X(x,z),
\quad
X(x+d,z) = X(x,z),
\ees
where $X$ is the chirality ``envelope,'' and then conducts regular
perturbation theory. As in our previous work
\cite{Nicholls19b,NichollsVo23,NichollsVo24} we will show that the
scattered field depends \textit{analytically} upon the parameter $\delta$
by directly estimating the $m$--th Taylor correction in appropriate
Sobolev spaces. By modifying the nature of our deformation slightly
we will rigorously demonstrate that this region of analyticity
contains a neighborhood of the \textit{entire} real axis (giving an
analytic continuation of our results), and that the field depends
\textit{jointly} analytically in both perturbation and spatial variables.
Finally, we will show how the resulting recursions
can be numerically approximated to deliver a numerical algorithm
of impressive accuracy and reliability.

The paper is organized as follows: In \S~\ref{Sec:Govern} we present
the governing equations with a discussion of the DBF constitutive
relations that model the chiral effects. In \S~\ref{Sec:TransBC} we
describe the transparent boundary conditions we utilize to not only
enforce the outgoing nature of the scattered waves, but also truncate
the problem to one of finite extent. In \S~\ref{Sec:Chiral} we outline
considerations special to chiral layers including a discussion of
decoupling the governing equations by polarization in \S~\ref{Sec:Decouple},
the Beltrami equations which govern these polarizations in 
\S~\ref{Sec:Transverse}, and the scalar Helmholtz equations which
the transverse components satisfy in \S~\ref{Sec:2DChirality}. 
In \S~\ref{Sec:HOPE} we derive our HOPE recursions in this context,
while, in preparation for our theoretical developments, we define our
function spaces in \S~\ref{Sec:FcnSpaces}. We rigorously establish analyticity
and analytic continuation in \S~\ref{Sec:AnalCont}, while we prove
joint analyticity in \S~\ref{Sec:Joint}. We close with our numerical results
in \S~\ref{Sec:NumRes} with discussion of our implementation in
\S~\ref{Sec:Implementation} and the triply layered medium problem we consider
in \S~\ref{Sec:LayMedia}. In \S~\ref{Sec:Conv} we present results of
the approximations produced by our implementation as compared to the
exact solution of the triply layered model. While this shows convergence,
we revisited these in \S~\ref{Sec:ConvSmooth} in the context of smoothed
solutions to display the aspirational exponential rate of convergence of
our algorithm. In \S~\ref{Sec:Conc} we give concluding remarks and
future directions.

%
%

\section{Governing Equations}
\label{Sec:Govern}

We consider structures whose response to illumination by
electromagnetic radiation is modeled by the
the time--harmonic Maxwell equations (time dependence 
$\exp(-i \omega t)$ factored out), in the absence of
currents and sources,
\begin{gather*}
\Curl{E} - i \omega B = 0, 
\quad
\Div{B} = 0, \\
\Curl{H} + i \omega D = 0,
\quad
\Div{D} = 0,
\end{gather*}
among the electric displacement, $D$, the magnetic
induction, $B$, the electric field, $E$, and the
magnetic field, $H$, \cite{Jackson75}.
To both model the effects of chirality and close
the Maxwell equations, we specify the Drude--Born--Fedorov (DBF)
constitutive relations
\bse
\label{Eqn:DBF:Constit}
\begin{align}
D = \epsilon(x,y,z) \left\{ E + \chi(x,y,z) \Curl{E} \right\}, \\
B = \mu(x,y,z) \left\{ H + \chi(x,y,z) \Curl{H} \right\},
\end{align}
\ese
\cite{LVV89Book,Lakhtakia94}. These equations feature
parameters which measure the electric, magnetic, and chiral responses,
namely the permittivity, $\epsilon(x,y,z)$, the permeability,
$\mu(x,y,z)$, and the chirality, $\chi(x,y,z)$. With these
the Maxwell equations now read (with the $(x,y,z)$ dependence
suppressed)
\begin{align*}
\Curl{E} - i \omega \mu \left\{ H + \chi \Curl{H} \right\} = 0, \\
\Curl{H} + i \omega \epsilon \left\{ E + \chi \Curl{E} \right\} = 0.
\end{align*}

It is very useful to find formulas for the curls of the electric and
magnetic fields in terms of these fields themselves. Writing these
Maxwell equations as
\bes
\begin{pmatrix} 1 & -i \omega \mu \chi \\
    i \omega \epsilon \chi & 1 \end{pmatrix}
\begin{pmatrix} \Curl{E} \\ \Curl{H} \end{pmatrix}
= \begin{pmatrix} i \omega \mu H \\ -i \omega \epsilon E \end{pmatrix},
\ees
we can solve for $\Curl{E}$ and $\Curl{H}$, yielding
\be
\label{Eqn:DBF:Curl}
\begin{pmatrix} \Curl{E} \\ \Curl{H} \end{pmatrix}
= J \begin{pmatrix} E \\ H \end{pmatrix},
\quad
J := \frac{1}{(1 - \chi^2 k^2)} \begin{pmatrix} \chi k^2 & i \omega \mu \\
  -i \omega \epsilon & \chi k^2 \end{pmatrix},
\ee
where $k^2 = \omega^2 \epsilon \mu$,
which we term the \textit{Maxwell--DBF equation}.
For obvious reasons we now make the assumption that 
$(1 - \chi^2 k^2) \neq 0$ which is typically enforced with
\be
\label{Eqn:ChiK}
-1 < \chi k < 1.
\ee

The permittivity, permeability, and chirality are all taken
to be $d_x \times d_y$ bi--periodic
\bes
\{ \epsilon, \mu, \chi \}(x,y,z)
  = \begin{cases} \{ \epsu, \muu, 0 \}, & z > h, \\
  \{ \epsv, \muv, \chiv \}(x,y,z), & -h < z < h, \\
  \{ \epsw, \muw, 0\}, & z < -h,
  \end{cases}
\ees
where $\epsu, \epsw, \muu, \muw \in \Real^+$,
\bes
\{ \epsv, \muv, \chiv \}(x+d_x,y+d_y,z)
  = \{ \epsv, \muv, \chiv \}(x,y,z),
\ees
and
\begin{align*}
\lim_{z \rightarrow h-} \{ \epsv, \muv, \chiv \}(x,y,z)
  & = \{ \epsu, \muu, 0 \}, \\
\lim_{z \rightarrow (-h)+} \{ \epsv, \muv, \chiv \}(x,y,z)
  & = \{ \epsw, \muw, 0 \}.
\end{align*}
For future use $(\km)^2 = \epsm \mum \omega^2$, $m \in \{ u, v, w \}$.

As in our previous contributions, we consider such structures
illuminated from above by plane--wave incident radiation of the
form
\bes
E^{\text{inc}}(x,y,z) = \AAEE
  \exp(i \alpha x + i \beta y - i \gammau z),
\quad
H^{\text{inc}}(x,y,z) = \AAHH
  \exp(i \alpha x + i \beta y - i \gammau z),
\ees
where
\bes
\Abs{\AAEE} = \Abs{\AAHH} = 1.
\ees
In order to solve the governing equations in the achiral
($\chi=0$) upper layer, $\{ z > h \}$, we ask that
\bes
\uk \times \AAEE = i \omega \muu \AAHH,
\quad
\uk \times \AAHH = -i \omega \epsu \AAEE,
\quad
\uk \cdot \AAEE = 0,
\quad
\uk \cdot \AAHH = 0,
\ees
where
\bes
\uk = \begin{pmatrix} \alpha \\ \beta \\ -\gammau \end{pmatrix}
  = \ku \begin{pmatrix} \sin(\theta) \cos(\phi) \\
  \sin(\theta) \sin(\phi) \\ \cos(\theta) \end{pmatrix},
\ees
and $(\phi,\theta)$ are the angles of incidence.

%
%
	
\section{Transparent Boundary Conditions}
\label{Sec:TransBC}

We now give a brief reprise of our description in
\cite{Nicholls19b,NichollsVo23,NichollsVo24} of the transparent boundary 
conditions we enforce. (See also the text of Bao \& Li \cite{BaoLi22}.)
These allow us to not only reduce the infinite domain
to one of finite size, but also rigorously specify far--field boundary
conditions. In the upper domain $\{ z > h \}$ we seek a solution 
in the sum of the incident radiation and an upward propagating
(reflected) component, e.g.,
\begin{align}
E & = E^{\text{inc}} + E^{\text{refl}} \notag \\
  & = \AAEE \exp(i \alpha x + i \beta y - i \gammau z) 
    + c_{0,0} \exp(i \alpha x + i \beta y + i \gammau (z-h)) \notag \\
  & \quad 
    + \sum_{(p,q) \neq (0,0)} \hat{u}_{p,q} \exp(i \alpha_p x 
	+ i \beta_q y + i \gammau_{p,q} (z-h)),
\label{Eqn:Rayleigh:u}
\end{align}
\cite{Petit80,Yeh05} where
\begin{gather*}
\alpha_p = \alpha + (2 \pi/d_x) p,
\quad
\beta_q = \beta + (2 \pi/d_y) q,
\\
\gammam_{p,q} := \begin{cases}
  \sqrt{(\km)^2 - \alpha_p^2 - \beta_q^2},
  & \alpha_p^2 + \beta_q^2 \leq \epsilon^{(m)} k_0^2, \\
  i \sqrt{\alpha_p^2 + \beta_q^2 - (\km)^2},
  & \alpha_p^2 + \beta_q^2 > \epsilon^{(m)} k_0^2,
\end{cases}
\quad
m \in \{ u, w \},
\end{gather*}
since $\epsu, \epsw, \muu, \muw \in \Real^+$.
By choosing $c_{0,0} = \hat{u}_{0,0} - \AAEE \exp(-i \gammau h)$, 
which implies that $\hat{u}_{0,0} = c_{0,0} + \AAEE \exp(-i \gammau h)$,
we have
\begin{align*}
E & = \AAEE \exp(i \alpha x + i \beta y - i \gammau z) 
  - \AAEE \exp(i \alpha x + i \beta y + i \gammau (z-2h)) \\
  & \quad
  + \sump \sumq \hat{u}_{p,q} 
  \exp(i \alpha_p x + i \beta_q y + i \gammau_{p,q} (z-h)),
\end{align*}
and $E(x,y,h) = u(x,y)$. Now, one can show that 
\begin{align*}
\pz E & = (-i \gammau) \AAEE \exp(i \alpha x + i \beta y - i \gammau z) \\
  & \quad 
  - (i \gammau) \AAEE \exp(i \alpha x + i \beta y + i \gammau (z-2h)) \\
  & \quad
  + \sump \sumq (i \gammau_{p,q}) \hat{u}_{p,q} 
  \exp(i \alpha_p x + i \beta_q y + i \gammau_{p,q} (z-h)),
\end{align*}
so that
\begin{align*}
-\pz E(x,y,h) 
  & = (i \gammau) \AAEE \exp(i \alpha x + i \beta y - i \gammau h) \\
  & \quad
  + (i \gammau) \AAEE \exp(i \alpha x + i \beta y + i \gammau (-h)) \\
  & \quad
  + \sump \sumq (-i \gammau_{p,q}) \hat{u}_{p,q} 
  \exp(i \alpha_p x + i \beta_q y).
\end{align*}
Defining the function
\be
\label{Eqn:phi:Def}
\phi(x,y) := \left( 2 i \gammau \exp(-i \gammau h) \right) \AAEE 
  \exp(i \alpha x + i \beta y),
\ee
and the order--one Fourier multiplier (the externally directed
Dirichlet--Neumann operator for the Maxwell equation on $\{ z > h \}$)
\bes
T_u[ \psi ] := \sump \sumq (-i \gammau_{p,q}) 
  \hat{\psi}_{p,q} \exp(i \alpha_p x + i \beta_q y),
\ees
then we see that we can express the Upward Propagating Condition (UPC)
\cite{ArensHab} exactly with the boundary condition
\bes
-\pz E(x,y,h) - T_u[ E(x,y,h) ] = \phi(x,y).
\ees
In an analogous fashion we discover that
\bes
-\pz H(x,y,h) - T_u[ H(x,y,h) ] = \psi(x,y),
\ees
where
\be
\label{Eqn:psi:Def}
\psi(x,y) := \left( 2 i \gammau \exp(-i \gammau h) \right) \AAHH 
  \exp(i \alpha x + i \beta y).
\ee
	
Similarly, in $\{ z < -h \}$ we seek a solution
which is purely downward propagating (transmitted)
\be
\label{Eqn:Rayleigh:w}
E = E^{\text{trans}}
  = \sump \sumq \hat{w}_{p,q} 
  \exp(i \alpha_p x + i \beta_q y - i \gammaw_{p,q} (z+h)),
\ee
\cite{Petit80,Yeh05}.
Clearly $E(x,y,-h) = w(x,y)$ and, using
\bes
\pz E(x,y,-h) = \sump \sumq (-i \gammaw_{p,q}) \hat{w}_{p,q}
  \exp(i \alpha_p x + i \beta_q y),
\ees
we define the analogous order--one Fourier multiplier (again, the
externally directed Dirichlet--Neumann operator for the Maxwell 
equation on $\{ z < -h \}$)
\bes
T_w[ \psi ] := \sump \sumq (-i \gammaw_{p,q}) \hat{\psi}_{p,q}
  \exp(i \alpha_p x + i \beta_q y).
\ees
We can state the Downward Propagating Condition (DPC) \cite{ArensHab}
transparently using
\bes
\pz E(x,y,-h) - T_w[ E(x,y,-h) ] = 0,
\ees
and
\bes
\pz H(x,y,-h) - T_w[ H(x,y,-h) ] = 0.
\ees

Gathering our full set of governing equations we find
the following problem to solve,
\bse
\label{Eqn:DBF}
\begin{align}
& \begin{pmatrix} \Curl{E} \\ \Curl{H} \end{pmatrix}
  = J \begin{pmatrix} E \\ H \end{pmatrix},
  && -h < z < h, \label{Eqn:DBF:a} \\
& -\pz E - T_u[ E ] = \phi, 
  \quad
  -\pz H - T_u[ H ] = \psi,
  && z = h,
  \label{Eqn:DBF:b} \\
& \pz E - T_w[ E ] = 0,
  \quad
  \pz H - T_w[ H ] = 0,
  && z = -h,
  \label{Eqn:DBV:c} \\
& E(x+d_x,y+d_y,z) = \exp(i \alpha d_x + i \beta d_y) E(x,y,z), 
  \label{Eqn:DBF:d} \\
& H(x+d_x,y+d_y,z) = \exp(i \alpha d_x + i \beta d_y) H(x,y,z).
  \label{Eqn:DBF:e}
\end{align}
\ese

%
%

\section{A Chiral Slab in Vacuum}
\label{Sec:Chiral}

We begin our study by considering a periodic chiral slab
in a vacuum where we have constant values of the permittivity and 
permeability
\bes
\epsm = \epsz,
\quad
\mum = \muz,
\quad
\km = \kz = \sqrt{\epsz \muz} \omega,
\quad
m \in \{ u, w \}.
\ees
Additionally we choose $\epsv = \epsz$ and $\muv = \muz$
for simplicity, though, any constant value can be accommodated.
As we will see, this assumption enables a change of variables
which not only delivers a significantly
simpler set of PDEs, but also reveals a particularly insightful
physical interpretation.

%
%

\subsection{Decoupling by Polarization}
\label{Sec:Decouple}

In this case of \textit{constant} values of the permeability
and permittivity, we can decouple the solutions of the Maxwell
equations by polarization (left-- and right--handed).
To demonstrate the decoupling we appeal to the Maxwell--DBF
equations in $\{ -h < z < h \}$, \eqref{Eqn:DBF:a},
in this setting
\bes
\begin{pmatrix} \Curl{E} \\ \Curl{H} \end{pmatrix}
= \Jz \begin{pmatrix} E \\ H \end{pmatrix},
\quad
\Jz := \frac{1}{(1 - \chi^2 \kz^2)}
  \begin{pmatrix} \chi \kz^2 & i \omega \muz \\
    -i \omega \epsz & \chi \kz^2 \end{pmatrix},
\ees
where we enforce $-1 < \chi \kz < 1$, c.f. \eqref{Eqn:ChiK}.
To decouple this system, we compute the eigenvalues and eigenvectors
of the matrix $\Jz$, which are readily shown to be
\bes
\lambda_L = \frac{\kz (\chi \kz + 1)}{1 - \chi^2 \kz^2}
  = \left( \frac{\kz}{1 - \chi \kz} \right) =: \kappa^L,
\quad
\xi^L = \begin{pmatrix} 1 \\ -i/\etaz \end{pmatrix},
\ees
where $\etaz = \sqrt{\muz/\epsz}$ is the vacuum impedance, and
\bes
\lambda_R = \frac{\kz (\chi \kz - 1)}{1 - \chi^2 \kz^2}
  = - \left( \frac{\kz}{1 + \chi \kz} \right) =: -\kappa^R,
\quad
\xi^R = \begin{pmatrix} -i \etaz \\ 1 \end{pmatrix},
\ees
and $\kz^2 = \kappa^L \kappa^R$.
With these we can form the diagonal and diagonalizing matrices
\bes
\Lambda = \begin{pmatrix} \kappa^L & 0 \\
    0 & -\kappa^R \end{pmatrix},
\quad
\Vz = \begin{pmatrix} 1 & -i \etaz \\ -i/\etaz & 1 \end{pmatrix},
\quad
\Vz^{-1} = \begin{pmatrix} 1/2 & i \etaz/2 \\ 
  i/(2 \etaz) & 1/2 \end{pmatrix}.
\ees
Setting
\be
\label{Eqn:Bohren}
\begin{pmatrix} L \\ R \end{pmatrix}
= \Vz^{-1} \begin{pmatrix} E \\ H \end{pmatrix}
\quad
\iff
\quad
\begin{pmatrix} E \\ H \end{pmatrix} 
= \Vz \begin{pmatrix} L \\ R \end{pmatrix},
\ee
we have the Bohren Transformation \cite{LVV89Book,Lakhtakia94}.
With these, it is not difficult to show that $L$ and $R$ satisfy,
\be
\label{Eqn:Beltrami}
\Curl{L} = \kappa^L L, 
\quad
\Curl{R} = -\kappa^R R, 
\ee
respectively, establishing them as Beltrami Fields \cite{Lakhtakia94}.
The suggestive notation now becomes clear: The field 
$L$ is Left Circularly Polarized (LCP) while $R$ is Right Circularly
Polarized (RCP). We notice that this change of variables is independent
of the chirality, $\chi$, which could be zero, non--zero, or even
non--constant. This gives the classical observation that all
solutions of Maxwell's equations in a homogeneous
medium are composed of LCP and RCP plane--waves.

In the case of an achiral material, $\chi \equiv 0$, we
have that $\kappa^L = \kappa^R = k_0$ and both polarities
travel at the \textit{same} speed. However, for chiral materials,
$\chi \neq 0$, $\kappa^L \neq \kappa^R$ and the polarizations
travel at \textit{different} speeds; this is circular dichromism (CD).

%
%

\subsection{Transverse Solution: Beltrami Equations}
\label{Sec:Transverse}

While not immediately apparent, the Beltrami equations, 
\eqref{Eqn:Beltrami}, admit transverse solutions in particular
configurations. To illustrate this we consider the generic 
variable ($y$--invariant) coefficient Beltrami equation
\bes
\Curl{Q} = \frac{\sigmaz}{\sigma(x,z)} Q,
\ees
where the particular form will become clear presently,
and consider $y$--invariant solutions of the form
\bes
Q = Q(x,z) = \begin{pmatrix} q^x(x,z) \\ q(x,z) \\ q^z(x,z) \end{pmatrix}.
\ees
It is not difficult to see that the Beltrami equation demands
\begin{align*}
    & -\pz q(x,z) = \sigmaz \left( \frac{q^x(x,z))}{\sigma(x,z)} \right), \\
    & \pz q^x(x,z) -\px q^z(x,z) 
    = \sigmaz \left( \frac{q(x,z)}{\sigma(x,z)} \right), \\
    & \px q(x,z) = \sigmaz \left( \frac{q^z(x,z)}{\sigma(x,z)} \right).
\end{align*}
The first and third equations can be readily solved in terms of $q(x,z)$,
\bes
q^x(x,z) = -\frac{\sigma(x,z) \pz q(x,z)}{\sigmaz},
\quad
q^z(x,z) = \frac{\sigma(x,z) \px q(x,z)}{\sigmaz},
\ees
and insertion of these into the second equation yields
\bes
\pz \left[ -\frac{\sigma(x,z) \pz q(x,z)}{\sigmaz} \right]
- \px \left[ \frac{\sigma(x,z) \px q(x,z)}{\sigmaz} \right] 
= \sigmaz \frac{q(x,z)}{\sigma(x,z)},
\ees
or
\be
\label{Eqn:Beltrami:Scalar}
\sigma(x,z) \Div{ \sigma(x,z) \Grad{q(x,z)} } + \sigmaz^2 q(x,z) = 0.
\ee
Naturally, if $\sigma \equiv 1$ then we realize a standard
Helmholtz equation
\bes
\Laplacian{q(x,z)} + \sigmaz^2 q(x,z) = 0,
\ees
however, even if $\sigma = \sigma(x,z)$, \eqref{Eqn:Beltrami:Scalar} 
is a \textit{scalar} PDE of
elliptic type for which our techniques are applicable.

%
%

\subsection{Two--Dimensional Chirality}
\label{Sec:2DChirality}

With these developments, it is easy to see that if the chirality
is $y$--independent, $\chi = \chi(x,z)$, and $x$--periodic (with period
$d_x = d$), then the fields $\{ L, R \}$ can be decomposed as
\bes
L = \begin{pmatrix} \ell^x(x,z) \\ \ell(x,z) \\ \ell^z(x,z) \end{pmatrix},
\quad
R = \begin{pmatrix} r^x(x,z) \\ r(x,z) \\ r^z(x,z) \end{pmatrix},
\ees
with the \textit{transverse} components satisfying the following
elliptic PDEs. Using $\sigmaz = \kz$ and $\sigma(x,z) = 1 - \kz \chi(x,z)$
for the LCP field we find
\be
(1 - \kz \chi(x,z)) \Div{(1 - \kz \chi(x,z)) \Grad{\ell(x,z)}} 
  + \kz^2 \ell(x,z) = 0.
\ee
Using $\sigmaz = -\kz$ and $\sigma(x,z) = 1 + \kz \chi(x,z)$
for the RCP field we find
\be
(1 + \kz \chi(x,z)) \Div{(1 + \kz \chi(x,z)) \Grad{r(x,z)}} 
  + \kz^2 r(x,z) = 0.
\ee
In this case of transverse polarization we realize that there
are two \textit{decoupled} PDEs to solve. The first,
using the (inverse) Bohren transformation \eqref{Eqn:Bohren}
for the data at $\Gamma_h$, is for the LCP field,
\bse
\label{Eqn:LCP}
\begin{align}
& (1 - \kz \chiv(x,z)) \Div{(1 - \kz \chiv(x,z)) \Grad{\ell(x,z)}} 
  + \kz^2 \ell(x,z) = 0,
  && -h < z < h, \label{Eqn:LCP:a} \\
& -\pz \ell - T_u[ \ell ] = \frac{1}{2} ( \phi + i \etaz \psi ),
  && z = h,
  \label{Eqn:LCP:b} \\
& \pz \ell - T_w[ \ell ] = 0,
  && z = -h,
  \label{Eqn:LCP:c} \\
& \ell(x+d,z) = \exp(i \alpha d) \ell(x,z). 
  \label{Eqn:LCP:d}
\end{align}
\ese
The second, again via the (inverse) Bohren transformation \eqref{Eqn:Bohren}
for the condition at $\Gamma_h$, is for the RCP field,
\bse
\label{Eqn:RCP}
\begin{align}
& (1 + \kz \chiv(x,z)) \Div{(1 + \kz \chiv(x,z)) \Grad{r(x,z)}} 
  + \kz^2 r(x,z) = 0,
  && -h < z < h, \label{Eqn:RCP:a} \\
& -\pz r - T_u[ r ] = \frac{1}{2} \left( \frac{i \phi}{\etaz} + \psi \right), 
  && z = h,
  \label{Eqn:RCP:b} \\
& \pz r - T_w[ r ] = 0,
  && z = -h,
  \label{Eqn:RCP:c} \\
& r(x+d,z) = \exp(i \alpha d) r(x,z).
  \label{Eqn:RCP:d}
\end{align}
\ese
To aid future developments we give a unified statement of
\eqref{Eqn:LCP} and \eqref{Eqn:RCP} as
\bse
\label{Eqn:LCPRCP}
\begin{align}
& \rho(x,z) \Div{ \rho(x,z) \Grad{v(x,z)}} + \kz^2 v(x,z) = 0,
  && \text{in $S_v$}, \label{Eqn:LCPRCP:a} \\
& -\pz v - T_u[ v ] = \tau, 
  && \text{at $\Gamma_h$},
  \label{Eqn:LCPRCP:b} \\
& \pz v - T_w[ v ] = 0,
  && \text{at $\Gamma_{-h}$},
  \label{Eqn:LCPRCP:c} \\
& v(x+d,z) = \exp(i \alpha d) v(x,z).
  \label{Eqn:LCPRCP:d}
\end{align}
\ese
where
\bes
v(x,z) = \begin{cases} \ell(x,z), & \text{LCP}, \\ 
  r(x,z), & \text{RCP}, \end{cases}
\quad
\rho(x,z) = \begin{cases} 1 - \kz \chiv(x,z), & \text{LCP}, \\
  1 + \kz \chiv(x,z), & \text{RCP}, \end{cases}
\ees
and
\be
\label{Eqn:tau}
\tau = \begin{cases}
  \frac{1}{2} ( \phi + i \etaz \psi ), & \text{LCP}, \\
  \frac{1}{2} ( \frac{i \phi}{\etaz} + \psi ), & \text{RCP},
  \end{cases}
\ee
and
\bes
S_v := (0,d) \times (-h,h),
\quad
\Gamma_{\pm h} := (0,d) \times \{ z = \pm h \}.
\ees

%
%

\section{A High--Order Perturbation of Envelopes Method}
\label{Sec:HOPE}

Following the lead of our previous work on the HOPE approach
to the achiral Helmholtz and Maxwell equations
\cite{Nicholls19b,NichollsVo23,NichollsVo24},
we suppose that the chirality is an (initially) small 
perturbation of a base value, $\bchi$,
\bes
\chiv(x,z) = \bchi - \delta X(x,z),
\ees
which has a slightly different form than
\cite{Nicholls19b,NichollsVo23,NichollsVo24} in order to
accommodate the case that, perhaps, $\bchi=0$. This gives
\begin{gather*}
\rho(x,z) = \rho_0 + \delta \rho_1(x,z),
\\
\rho_0 = \begin{cases} 1 - \kz \bchi, & \text{LCP}, \\
  1 + \kz \bchi, & \text{RCP}, \end{cases}
\quad
\rho_1(x,z) = \begin{cases} \kz X(x,z), & \text{LCP}, \\
  -\kz X(x,z), & \text{RCP}, \end{cases}
\end{gather*}
and, as above, $-1 < \bchi \kz < 1$, c.f.\ \eqref{Eqn:ChiK}.
With this \eqref{Eqn:LCPRCP} becomes
\bse
\label{Eqn:LCPRCP:delta}
\begin{align}
& \rho_0 \Div{ \rho_0 \Grad{v(x,z)} } + \kz^2 v(x,z) = F(x,z), 
  && \text{in $S_v$}, \label{Eqn:LCPRCP:delta:a} \\
& -\pz v - T_u[ v ] = \tau, 
  && \text{at $\Gamma_h$},
  \label{Eqn:LCPRCP:delta:b} \\
& \pz v - T_w[ v ] = 0,
  && \text{at $\Gamma_{-h}$},
  \label{Eqn:LCPRCP:delta:c} \\
& v(x+d,z) = \exp(i \alpha d) v(x,z).
  \label{Eqn:LCPRCP:delta:d}
\end{align}
where
\begin{align}
F(x,z) & = -\delta \rho_0 \Div{ \rho_1(x,z) \Grad{v(x,z)} }
  - \delta \rho_1(x,z) \Div{ \rho_0 \Grad{v(x,z)} } \notag \\
& \quad 
  - \delta^2 \rho_1(x,z) \Div{ \rho_1(x,z) \Grad{v(x,z)} }.
\end{align}
\ese

We now expand
\be
\label{Eqn:v:Exp}
v(x,z;\delta) = \summ v_m(x,z) \delta^m,
\ee
and insert this into \eqref{Eqn:LCPRCP:delta} giving
the recursive problems for $v_m(x,z)$ as
\bse
\label{Eqn:LCPRCP:m}
\begin{align}
& \rho_0 \Div{ \rho_0 \Grad{v_m(x,z)}} + \kz^2 v_m(x,z) = F_m(x,z), 
  && \text{in $S_v$}, \label{Eqn:LCPRCP:m:a} \\
& -\pz v_m - T_u[ v_m ] = \delta_{m,0} \tau,
  && \text{at $\Gamma_h$},
  \label{Eqn:LCPRCP:m:b} \\
& \pz v_m - T_w[ v_m ] = 0,
  && \text{at $\Gamma_{-h}$},
  \label{Eqn:LCPRCP:m:c} \\
& v_m(x+d,z) = \exp(i \alpha d) v_m(x,z),
  \label{Eqn:LCPRCP:m:d}
\end{align}
where, $\delta_{m,n}$ is the Kronecker delta function, and
\begin{align}
F_m(x,z) & = -\rho_0 \Div{ \rho_1(x,z) \Grad{v_{m-1}(x,z)} }
  - \rho_1(x,z) \Div{ \rho_0 \Grad{v_{m-1}(x,z)} } \notag \\
& \quad 
  - \rho_1(x,z) \Div{ \rho_1(x,z) \Grad{v_{m-2}(x,z)} }.
  \label{Eqn:LCPRCP:m:e}
\end{align}
\ese

We note that we can solve each of these at order zero with plane 
wave solutions
\bes
v_0(x,z) = A \exp(i \alpha x - i \bar{\gamma} z),
\quad
\alpha^2 + \bar{\gamma}^2 = \frac{\kz^2}{\rho_0^2}.
\ees
A helpful point of view of our HOPE scheme is that we seek higher
order corrections to these solutions
\bes
v(x,z) = v_0(x,z) + \sum_{m=1}^{\infty} v_m(x,z) \delta^m.
\ees

To specify more clearly the nature of our algorithm we must clarify
the form of the envelope, $X(x,z)$. For the numerical experiments
of \S~\ref{Sec:NumRes} we follow the lead of our previous work
\cite{Nicholls19b} and choose one meant to emulate a solid layer
of chiral material. Consider the smooth bump function
\bes
\Phi_{a,b}(z) := \frac{\tanh(w(z-a))-\tanh(w(z-b))}{2},
\ees
with sharpness parameter $w$, which is essentially unity on $(a,b)$
while being effectively zero outside. We desire a uniform chiral layer of
thickness $2 t$ and chirality $\chi'$ in an otherwise achiral material
by choosing
\be
\label{Eqn:X:Smooth}
\bchi = 0,
\quad
X(x,z) = \chi' \Phi_{-t,t}(z),
\quad
\delta = -1.
\ee
In Figure~\ref{Fig:Chirality} we plot the resulting $\chiv(x,z)$
for $d = 0.8$, $h = 0.95$, $t = 0.5$, $\chi' = 0.01$, and $w = 100$.
%
%
\begin{figure}[hbt]
  \begin{center}
  \includegraphics[width=0.6\textwidth]{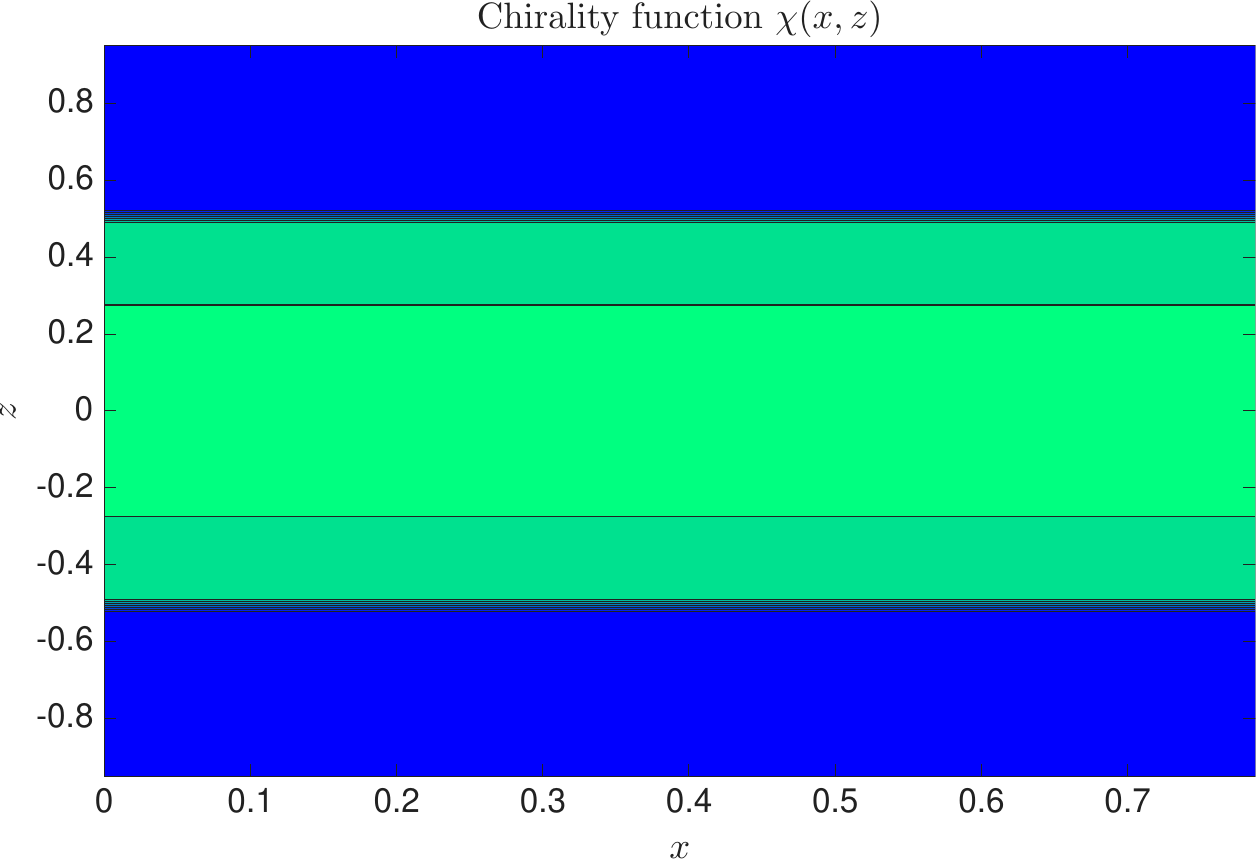}
  \caption{Plot of an example chirality function, 
  $\chiv(x,z)$,
  approximating a chiral slab of unit thickness.}
\label{Fig:Chirality}
\end{center}
\end{figure}
%
%
%

%
%

\section{Function Spaces}
\label{Sec:FcnSpaces}

We now recall function spaces and theoretical
results that are necessary for our analysis,
very much in the spirit of our previous work
\cite{Nicholls19b,NichollsVo23,NichollsVo24}. For any real number $s \geq 0$,
we have the interfacial quasiperiodic $L^2$ Sobolev norm
\bes
\SobNorm{v}{s}^2 := \sump \Angle{p}^{2 s} \Abs{\hat{v}_p}^2,
\ees
where
\bes
\Angle{p}^2 := 1 + \Abs{p}^2,
\quad
\hat{v}_p := \frac{1}{d} \int_0^{d} v(x) e^{-i \alpha_p x} \; dx.
\ees
From this we define the interfacial quasiperiodic Sobolev space
\cite{Kress14}
\bes
H^s(\Gamma) = \left\{ v(x) \in L^2(\Gamma)\ |\ 
  \SobNorm{v}{s} < \infty \right\},
\quad
\Gamma := (0,d).
\ees
In addition, we mention that the dual space of $H^s$, $H^{-s}$, can be
defined by the norm above with a negative index. We also recall
the space of $s$--times continuously differentiable 
functions with H\"older norm
\bes
\HolderNorm{g}{s} := \max_{0 \leq \ell + r \leq s} 
  \SupNorm{\partial_x^{\ell} \partial_z^r g(x,z)}.
\ees
Finally, we define the volumetric laterally quasiperiodic Sobolev 
space as
\bes
H^s(S_v) = \left\{ u(x,z) \in L^2(S_v)\ |\ 
  \SobNorm{u}{s} < \infty \right\},
\ees
where 
\bes
\SobNorm{u}{s}^2
  := \sum_{j=0}^{s} \sump \Angle{p}^{2 (s-j)} 
  \int_{-h}^h \Abs{\partial_z^j \hat{u}_p(z)}^2 \; dz.
\ees
	
We mention an important lemma \cite{Evans10,NichollsReitich99} 
required for our later proofs.
\begin{Lemma}
Let $s \geq 0$ be an integer and $D$ be either $\Gamma$ or $S_v$.
If $g \in C^s(D)$ and $w \in H^s(D)$ then $g w \in H^s(D)$ and
\bes
\SobNorm{g w}{s} \leq \tM(s,D) \HolderNorm{g}{s} \SobNorm{w}{s},
\ees
where $\tM$ is some positive constant.
\end{Lemma}

For our results on joint analyticity we require a particular notion
of analytic function.
\begin{Def}
\label{Def:Comega}
Given any integer $\ell \geq 0$, the functions $f = f(x)$ and
$X = X(x,z)$ are members of the spaces $C^{\omega}_{\ell}(\Gamma)$
and $C^{\omega}_{\ell}(S_v)$, respectively, if they are real
analytic and satisfy the estimates
\bes
\HolderNorm{\frac{\px^r}{r!} f}{\ell} \leq C_f \frac{A^r}{(r+1)^2},
\quad
\HolderNorm{\frac{\px^r \pz^t}{(r+t)!} X}{\ell} \leq C_X 
  \frac{A^r}{(r+1)^2} \frac{D^t}{(t+1)^2},
\quad
\forall\ r, t \geq 0,
\ees
for some $C_f, C_X, A, D > 0$.
\end{Def}

Finally, we recall the following elementary result
\cite{NichollsReitich00b,Nicholls19b}.
\begin{Lemma}
\label{Lemma:S}
Let $s \geq 0$ be an integer, then there exists a constant $S>0$ such that
\bes
\sum_{j=0}^s \frac{(s+1)^2}{(s-j+1)^2(j+1)^2} < S,
\quad
\sum_{j=0}^s \sum_{r=0}^j \frac{(s+1)^2}{(s-j+1)^2 (j-r+1)^2 (r+1)^2} < S^2.
\ees
\end{Lemma}

%
%

\section{Analytic Continuation}
\label{Sec:AnalCont}

As we will demonstrate shortly, in either left or right circular polarization,
the series \eqref{Eqn:v:Exp} is strongly convergent in an appropriate function
space. In theory, this demands that $\delta \ll 1$, which suffices for small
deviations of the chirality from its background value, however, with 
a little more effort
we can demonstrate that the field $v(x,z;\delta)$ is, in fact, analytic
for \textit{any} \textit{real} value of the perturbation parameter $\delta$.
We follow our previous work \cite{Nicholls19b,NichollsVo24} (inspired by the
proofs which can be found in \cite{NichollsReitich00b,NichollsTaber06}) and
conclude from this result that numerical analytic continuation algorithms,
such as Pad\'e approximation \cite{BakerGravesMorris96}, are fully justified.

To accomplish this proof, we consider the envelope $X(x,z)$ and an 
\textit{arbitrary} \textit{real} $\tdelta \in \Real$, and show
that $v$ depends \textit{analytically} upon $\tdelta$. Considering
our current framework we set
\bes
X_0(x,z) := \tdelta_0 X(x,z),
\quad
\delta = \tdelta - \tdelta_0,
\ees
and seek analyticity about $\delta = 0$ for
\bes
\tdelta X(x,z) = (\tdelta_0 + \delta) X(x,z) = X_0(x,z) + \delta X(x,z),
\ees
see Figure~\ref{Fig:DomAnal}.
%
%
\begin{figure}[hbt]
  \begin{center}
  \includegraphics[width=0.6\textwidth]{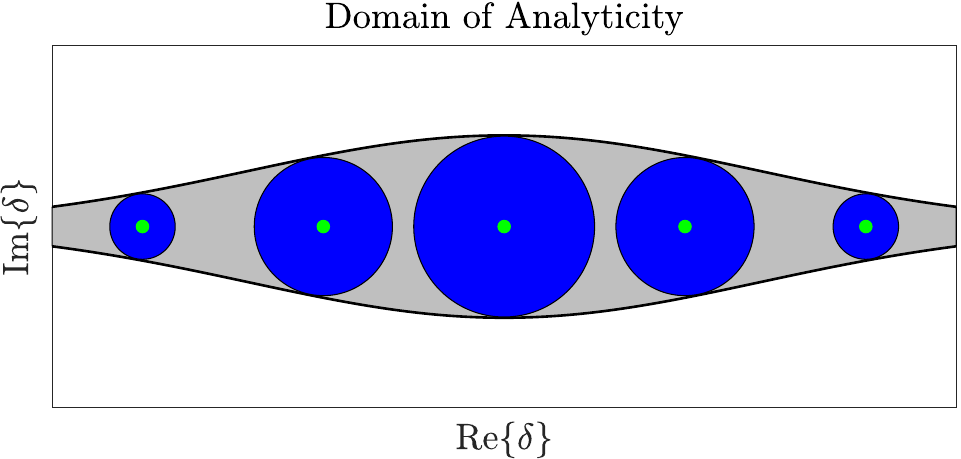}
  \caption{The domain of analyticity (gray) in the
  complex plane of the field, \eqref{Eqn:v:Exp},
  as a function of the perturbation parameter, $\delta$.
  We demonstrate existence of \textit{disks} (blue)
  of analyticity
  around arbitrary values $\tdelta_0 \in \Real$ (green) which 
  fills out the entire gray region.}
\label{Fig:DomAnal}
\end{center}
\end{figure}
With this, we have
\begin{gather*}
\rho(x,z) = \rho_0(x,z) + \delta \rho_1(x,z),
\\
\rho_0(x,z) = \begin{cases}
  1 - \kz \bchi + \kz X_0(x,z), \\
  1 + \kz \bchi - \kz X_0(x,z),
  \end{cases}
\quad
\rho_1(x,z) = \begin{cases}
  \kz X(x,z), & \text{LCP}, \\
  -\kz X(x,z), & \text{RCP},
  \end{cases}
\end{gather*}
so that \eqref{Eqn:LCPRCP:delta} can be updated to
\bse
\label{Eqn:LCPRCP:tdelta}
\begin{align}
& \rho_0(x,z) \Div{ \rho_0(x,z) \Grad{v(x,z)} } + \kz^2 v(x,z) = F(x,z), 
  && \text{in $S_v$}, \label{Eqn:LCPRCP:tdelta:a} \\
& -\pz v - T_u[ v ] = \tau, 
  && \text{at $\Gamma_h$},
  \label{Eqn:LCPRCP:tdelta:b} \\
& \pz v - T_w[ v ] = 0,
  && \text{at $\Gamma_{-h}$},
  \label{Eqn:LCPRCP:tdelta:c} \\
& v(x+d,z) = \exp(i \alpha d) v(x,z).
  \label{Eqn:LCPRCP:tdelta:d}
\end{align}
where
\begin{align}
F(x,z) & = -\delta \rho_0(x,z) \Div{ \rho_1(x,z) \Grad{v(x,z)} } \notag \\
& \quad
  - \delta \rho_1(x,z) \Div{ \rho_0(x,z) \Grad{v(x,z)} } \notag \\
& \quad 
  - \delta^2 \rho_1(x,z) \Div{ \rho_1(x,z) \Grad{v(x,z)} }.
\end{align}
\ese

As before, we expand
\be
\label{Eqn:v:tExp}
v(x,z;X_0,\delta) = \summ v_m(x,z;X_0) \delta^m,
\ee
and insert this into \eqref{Eqn:LCPRCP:tdelta} giving
the recursive problems for $v_m(x,z;X_0)$ as
\bse
\label{Eqn:LCPRCP:t:m}
\begin{align}
& \rho_0(x,z) \Div{ \rho_0(x,z) \Grad{v_m(x,z)}} + \kz^2 v_m(x,z) = F_m(x,z), 
  && \text{in $S_v$}, \label{Eqn:LCPRCP:t:m:a} \\
& -\pz v_m - T_u[ v_m ] = \delta_{m,0} \tau,
  && \text{at $\Gamma_h$},
  \label{Eqn:LCPRCP:t:m:b} \\
& \pz v_m - T_w[ v_m ] = 0,
  && \text{at $\Gamma_{-h}$},
  \label{Eqn:LCPRCP:t:m:c} \\
& v_m(x+d,z) = \exp(i \alpha d) v_m(x,z),
  \label{Eqn:LCPRCP:t:m:d}
\end{align}
where
\begin{align}
F_m(x,z) & = -\rho_0(x,z) \Div{ \rho_1(x,z) \Grad{v_{m-1}(x,z)} } \notag \\
& \quad
  - \rho_1(x,z) \Div{ \rho_0(x,z) \Grad{v_{m-1}(x,z)} } \notag \\
& \quad 
  - \rho_1(x,z) \Div{ \rho_1(x,z) \Grad{v_{m-2}(x,z)} }.
  \label{Eqn:LCPRCP:t:m:e}
\end{align}
\ese

As in \cite{Nicholls19b} we require an elliptic estimate for
our inductive proof which we report here. As \cite{BaoLi22} state, 
the issue of \textit{uniqueness} of solutions to the 
homogeneous Maxwell problem
\bse
\label{Eqn:Max:Uniqueness}
\begin{align}
& \rho_0(x,z) \Div{ \rho_0(x,z) \Grad{V} } + \kz^2 V = 0,
  && \text{in $S_v$}, \\
& -\pz V - T_u[V] = 0, && \text{at $\Gamma_h$}, \\
& \pz V - T_w[V] = 0, && \text{at $\Gamma_{-h}$}, \\
& V(x+d,z) = \exp(i \alpha d) V(x,z),
\end{align}
\ese
c.f.\ \eqref{Eqn:LCPRCP:t:m}, which should have only
the \textit{trivial} solution $V \equiv 0$, is a difficult one
and certain frequencies $\omega$ will induce non--uniqueness in 
some configurations. For better or worse, 
a precise characterization of the set of forbidden
frequencies is unknown and all that is clear is that it is countable
and accumulates at infinity \cite{BaoLi22}. To accommodate this
state of affairs we define the set of permissible configurations
\bes
\cP := \left\{ (\omega,\bchi,X_0)\ |\ 
V \equiv 0\ \text{is the unique solution of
	\eqref{Eqn:Max:Uniqueness}} \right\}.
\ees

With this, we can now state the following fundamental elliptic
regularity result.
\begin{Thm}
\label{Thm:EllEst}
Given any integer $s \geq 0$, if $(\omega,\bchi,X_0) \in \cP$,
$\rho_0 \in C^{s+1}(S_v)$,
$F \in H^{s}(S_v)$,
$Q \in H^{s+1/2}(\Gamma)$, and
$R \in H^{s+1/2}(\Gamma)$, 
then there exists a unique solution of
\bse
\label{Eqn:EllEst}
\begin{align}
& \rho_0(x,z) \Div{ \rho_0(x,z) \Grad{v}} + \kz^2 v = F,
  && \text{in $S_v$}, 
  \label{Eqn:EllEst:a} \\
& -\pz v - T_u[v] = Q && \text{at $\Gamma_h$}, \\
& \pz v - T_w[v] = R && \text{at $\Gamma_{-h}$}, \\
& v(x+d,z) = e^{i \alpha d} v(x,z),
\end{align}
\ese
satisfying
\be
\label{Eqn:EllEst:Est}
\SobNorm{v}{s+2} \leq C_e \left( \SobNorm{F}{s} 
  + \SobNorm{Q}{s+1/2} + \SobNorm{R}{s+1/2} \right),
\ee
where $C_e > 0$ is a constant.
\end{Thm}
\begin{proof}
This result is very much in the spirit of Theorem~6.2 in
\cite{Nicholls19b}, both of which can be established using
the methods found in \cite{BaoLi22}.
\end{proof}

We can now establish the recursive estimate required by our theory.
\begin{Lemma}
\label{Lemma:Recur}
Given any integer $s \geq 0$, if $\rho_0(x,z), \rho_1(x,z) \in C^{s+1}(S_v)$ and
\bes
\SobNorm{v_m}{s+2} \leq K B^m,
\quad
m < M,
\ees
then, for the functions $F_m$ in \eqref{Eqn:LCPRCP:t:m:e},
\bes
\SobNorm{F_M}{s} \leq K \tM^2
  \left\{ 2 \HolderNorm{\rho_0}{s+1} \HolderNorm{\rho_1}{s+1} B^{M-1} 
  + \HolderNorm{\rho_1}{s+1}^2 B^{M-2} \right\}.
\ees
\end{Lemma}
\begin{proof}
We begin our estimations with
\begin{align*}
\SobNorm{F_M}{s}
  & \leq \tM \HolderNorm{\rho_0}{s} \SobNorm{ \rho_1 \Grad{v_{M-1}} }{s+1}
  + \tM \HolderNorm{\rho_1}{s} \SobNorm{ \rho_0 \Grad{v_{M-1}} }{s+1} \\
  & \quad 
  + \tM \HolderNorm{\rho_1}{s} \SobNorm{ \rho_1 \Grad{v_{M-2}} }{s+1}.
\end{align*}
Continuing
\begin{align*}
\SobNorm{F_M}{s}
  & \leq \tM \HolderNorm{\rho_0}{s} \tM \HolderNorm{\rho_1}{s+1} \SobNorm{v_{M-1}}{s+2} 
  + \tM \HolderNorm{\rho_1}{s} \tM \HolderNorm{\rho_0}{s+1} \SobNorm{v_{M-1}}{s+2} \\
  & \quad 
  + \tM \HolderNorm{\rho_1}{s} \tM \HolderNorm{\rho_1}{s+1} \SobNorm{v_{M-2}}{s+2}.
\end{align*}
Finally, the inductive hypothesis yields
\bes
\SobNorm{F_M}{s}
  \leq 2 \tM^2 \HolderNorm{\rho_0}{s+1} \HolderNorm{\rho_1}{s+1} K B^{M-1}
  + \tM^2 \HolderNorm{\rho_1}{s+1}^2 K B^{M-2},
\ees
and we are done.
\end{proof}
	
We can now prove our analytic continuation result.
\begin{Thm}
\label{Thm:AnalCont}
Given any integer $s \geq 0$, if $(\omega,\bchi,X_0) \in \cP$,
and $\rho_0(x,z), \rho_1(x,z) \in C^{s+1}(S_v)$ then
the series \eqref{Eqn:v:tExp} converges strongly.
More precisely,
\be
\label{Eqn:AnalCont}
\SobNorm{v_m}{s+2} \leq K B^m,
\quad
\forall\ m \geq 0,
\ee
for some constants $K, B > 0$.
\end{Thm}
\begin{proof}
We prove the estimate \eqref{Eqn:AnalCont} by induction. From the definitions
of $\phi$ and $\psi$ it is clear that $\tau \in H^{s+1/2}$ for any $s \geq 0$.
For $m = 0$
we invoke Theorem~\ref{Thm:EllEst} with $F \equiv 0$, $Q = \tau$, and
$R \equiv 0$ and set $K := \SobNorm{v_0}{s+2}$. We now assume that
\eqref{Eqn:AnalCont} is true for all $m < M$ and use Theorem~\ref{Thm:EllEst}
with $Q \equiv R \equiv 0$ to deduce that
\bes
\SobNorm{v_M}{s+2} \leq C_e \SobNorm{F_M}{s}.
\ees
From Lemma~\ref{Lemma:Recur} we have that
\bes
\SobNorm{v_M}{s+2} \leq C_e K \tM^2
  \left\{ 2 \HolderNorm{\rho_0}{s+1} \HolderNorm{\rho_1}{s+1} B^{M-1} 
  + \HolderNorm{\rho_1}{s+1}^2 B^{M-2} \right\},
\ees
and we are done provided that
\bes
B \geq \max \left\{ 4 C_e \tM^2 \HolderNorm{\rho_0}{s+1},
  \sqrt{2 C_e} \tM \right\} \HolderNorm{\rho_1}{s+1}.
\ees
\end{proof}

\begin{Rk}
We point out that by setting $X_0 \equiv 0$ we obtain the classical
analyticity result about $\tdelta=0$.
\end{Rk}

\begin{Rk}
We observe that the choice of $B$ determines the radius of convergence
of the series \eqref{Eqn:v:tExp}, for instance, if we choose 
$B \delta \leq 1/2$ (implying $\delta \leq 1/(2 B)$) then we can guarantee,
from \eqref{Eqn:v:tExp}, that $\SobNorm{v}{s+2} < 2 K$. The dependence of
the crucial constant $B$ upon $C_e$ and $\HolderNorm{\rho_0}{s+1}$, recalling
\bes
\rho_0(x,z) = \begin{cases}
  1 - \kz \bchi + \kz X_0(x,z), & \text{LCP}, \\
  1 + \kz \bchi - \kz X_0(x,z), & \text{RCP},
  \end{cases}
\ees
is somewhat complicated, however, the role of $\tM \HolderNorm{\rho_1}{s+1}$,
remembering
\bes
\rho_1(x,z) = \begin{cases}
  \kz X(x,z), & \text{LCP}, \\
  -\kz X(x,z), & \text{RCP},
  \end{cases}
\ees
is clear: Linear dependence upon $\tM$ and 
$\HolderNorm{\rho_1}{s+1} = \kz \HolderNorm{X(x,z)}{s+1}$.
\end{Rk}

%
%

\section{Joint Analyticity}
\label{Sec:Joint}

We conclude our theory with a joint analyticity result in the spirit
of that found in \cite{Nicholls19b,NichollsVo23} which
demonstrates the \textit{joint} analyticity of $v = v(x,z;\delta)$
with respect to $x$, $z$, and $\delta$. We accomplish this by showing
that the $v_m(x,z)$ from \eqref{Eqn:v:tExp} satisfy
Definition~\ref{Def:Comega} which is
achieved by repeatedly differentiating the problem
\eqref{Eqn:LCPRCP:t:m}
with respect to $x$ and $z$.

This requires the generalization of our elliptic estimate,
Theorem~\ref{Thm:EllEst}, to the following result.
\begin{Thm}
\label{Thm:EllEst:Joint}
Given any integer $\ell \geq 1$, if $(\omega,\bchi,X_0) \in \cP$,
$\rho_0 \in C^{\omega}_{\ell}(S_v)$, so that
\bes
\HolderNorm{\frac{\px^r \pz^t}{(r+t)!} \rho_0}{\ell}
  \leq C_{\rho} \frac{A^r}{(r+1)^2} \frac{D^t}{(t+1)^2},
  \quad
  \forall\ r, t \geq 0,
\ees
and some $C_{\rho}, A, D > 0$, and $F \in C^{\omega}(S_v)$
satisfying
\bes
\SobNorm{\frac{\px^r \pz^t}{(r+t)!} F}{0}
  \leq C_F \frac{A^r}{(r+1)^2} \frac{D^t}{(t+1)^2},
\quad
\forall\ r, t \geq 0,
\ees
and some $C_F > 0$, and $Q \in C^{\omega}(\Gamma_h)$,
$R \in C^{\omega}(\Gamma_{-h})$ satisfying
\bes
\SobNorm{\frac{\px^r}{r!} Q}{1/2} \leq C_Q \frac{A^r}{(r+1)^2},
\quad
\SobNorm{\frac{\px^r}{r!} R}{1/2} \leq C_R \frac{A^r}{(r+1)^2},
\quad
\forall\ r \geq 0,
\ees
for some $C_Q, C_R > 0$, 
then there exists a unique solution $V \in C^{\omega}(S_v)$ of
\bse
\label{Eqn:EllEst:Joint}
\begin{align}
& \rho_0(x,z) \Div{ \rho_0(x,z) \Grad{V}} + \kz^2 V = F,
  && \text{in $S_v$}, 
  \label{Eqn:EllEst:Joint:a} \\
& -\pz V - T_u[V] = Q && \text{at $\Gamma_h$}, \\
& \pz V - T_w[V] = R && \text{at $\Gamma_{-h}$}, \\
& V(x+d,z) = e^{i \alpha d} V(x,z),
\end{align}
\ese
satisfying
\be
\label{Eqn:EllEst:Est:Joint}
\SobNorm{\frac{\px^r \pz^t}{(r+t)!} V}{2}
  \leq \underline{C}_e \frac{A^r}{(r+1)^2} \frac{D^t}{(t+1)^2},
\quad
\forall\ r, t \geq 0,
\ee
where
\bes
\underline{C}_e := \kappa(h) (C_F + C_Q + C_R),
\ees
for some constant $\kappa(h) > 0$.
\end{Thm}
\begin{proof}
The proof of this theorem is similar to that of Theorem 7.2
found in \cite{Nicholls19b}.
\end{proof}

We now establish a requisite recursive estimate.
\begin{Lemma}
\label{Lemma:Recur:Joint}
Given any integer $\ell \geq 1$, if 
$\rho_0(x,z), \rho_1(x,z) \in C^{\omega}_{\ell}(S_v)$
so that
\bes
\HolderNorm{\frac{\px^r \pz^t}{(r+t)!} \rho_0}{\ell}
  \leq C_{\rho} \frac{A^r}{(r+1)^2} \frac{D^t}{(t+1)^2},
  \quad
\HolderNorm{\frac{\px^r \pz^t}{(r+t)!} \rho_1}{\ell}
  \leq C_{\rho} \frac{A^r}{(r+1)^2} \frac{D^t}{(t+1)^2},
\ees
for all $r, t \geq 0$, for some $C_{\rho}, A, D > 0$, and
\bes
\SobNorm{\frac{\px^r \pz^t}{(r+t)!} v_m}{2}
  \leq \underline{K} B^m \frac{A^r}{(r+1)^2} \frac{D^t}{(t+1)^2},
\quad
\forall\ m < M,
\quad
\forall\ r, t \geq 0,
\ees
for some $\underline{K}, B > 0$ then
for the function $F_M$ in \eqref{Eqn:LCPRCP:t:m:e},
\bes
\SobNorm{\frac{\px^r \pz^t}{(r+t)!} F_M}{0}
  \leq \underline{\tilde{C}} \underline{K} B^{M-1} 
  \frac{A^r}{(r+1)^2} \frac{D^t}{(t+1)^2},
\ees
for some $\underline{\tilde{C}} > 0$.
\end{Lemma}
\begin{proof}
We recall that
\bes
F_M = -\rho_0 \Div{ \rho_1 \Grad{v_{M-1}} }
  - \rho_1 \Div{ \rho_0 \Grad{v_{M-1}} }
  - \rho_1 \Div{ \rho_1 \Grad{v_{M-2}} },
\ees
and focus on the first term (as the others are similar),
\bes
F_M^{(1)} = -\rho_0 \Div{ \rho_1 \Grad{v_{M-1}} },
\ees
so that
\begin{align*}
\frac{\px^r \pz^t}{(r+t)!} F_M^{(1)}
  & = -\frac{r! t!}{(r+t)!} \sum_{j=0}^{r} \sum_{k=0}^{j}
  \sum_{p=0}^{t} \sum_{q=0}^{p} 
  \left( \frac{\px^{r-j}}{(r-j)!} \frac{\pz^{t-p}}{(t-p)!} \rho_0 \right) \\
  & \quad \times
  \Div{ \left( \frac{\px^{j-k}}{(j-k)!} \frac{\pz^{p-q}}{(p-q)!} \rho_1 \right)
  \Grad{ \left\{ \frac{\px^k}{k!} \frac{\pz^q}{q!} v_{M-1}} \right\} }.
\end{align*}
Now, using $(r! t!) \leq (r+t)!$, we estimate
\begin{align*}
\SobNorm{ \frac{\px^r \pz^t}{(r+t)!} F_M^{(1)} }{0}
  & \leq \frac{r! t!}{(r+t)!} \sum_{j=0}^{r} \sum_{k=0}^{j}
  \sum_{p=0}^{t} \sum_{q=0}^{p} 
  \tM \HolderNorm{\frac{\px^{r-j}}{(r-j)!} \frac{\pz^{t-p}}{(t-p)!} \rho_0}{0} \\
  & \quad \times
  \tM \HolderNorm{\frac{\px^{j-k}}{(j-k)!} \frac{\pz^{p-q}}{(p-q)!} \rho_1}{1}
  \SobNorm{\frac{\px^k}{k!} \frac{\pz^q}{q!} v_{M-1}}{2} \\
  & \leq \sum_{j=0}^{r} \sum_{k=0}^{j}
  \sum_{p=0}^{t} \sum_{q=0}^{p} 
  \tM C_{\rho} \frac{A^{r-j}}{(r-j+1)^2} \frac{D^{t-p}}{(t-p+1)^2} \\
  & \quad \times
  \tM C_{\rho} \frac{A^{j-k}}{(j-k+1)^2} \frac{D^{p-q}}{(p-q+1)^2}
  \underline{K} B^{M-1} \frac{A^k}{(k+1)^2} \frac{D^q}{(q+1)^2}.
\end{align*}
Continuing
\begin{align*}
\SobNorm{ \frac{\px^r \pz^t}{(r+t)!} F_M^{(1)} }{0}
  & \leq \tM^2 C_{\rho}^2 \underline{K} B^{M-1} \frac{A^r}{(r+1)^2} \frac{D^t}{(t+1)^2} \\
  & \quad \times
  \sum_{j=0}^{r} \sum_{k=0}^{j} \frac{(r+1)^2}{(r-j+1)^2 (j-k+1)^2 (k+1)^2} \\
  & \quad \times
  \sum_{p=0}^{t} \sum_{q=0}^{p} \frac{(t+1)^2}{(t-p+1)^2 (p-q+1)^2 (q+1)^2} \\
  & \leq \tM^2 C_{\rho}^2 \underline{K} B^{M-1} \frac{A^r}{(r+1)^2} \frac{D^t}{(t+1)^2}
  S^2 S^2,
\end{align*}
and we are done provided that we choose
\bes
\underline{\tilde{C}} \geq \tM^2 C_{\rho}^2 S^4.
\ees
\end{proof}

We are now ready to state and establish our joint analyticity result.
\begin{Thm}
\label{Thm:AnalCont:Joint}
Given any integer $\ell \geq 1$, if $(\omega,\bchi,\rho_0) \in \cP$,
and $\rho_0(x,z), \rho_1(x,z) \in C^{\omega}_{\ell}(S_v)$ so that
\bes
\HolderNorm{\frac{\px^r \pz^t}{(r+t)!} \rho_0}{\ell}
  \leq C_{\rho} \frac{A^r}{(r+1)^2} \frac{D^t}{(t+1)^2},
\quad
\HolderNorm{\frac{\px^r \pz^t}{(r+t)!} \rho_1}{\ell}
  \leq C_{\rho} \frac{A^r}{(r+1)^2} \frac{D^t}{(t+1)^2},
\ees
for all $r, t \geq 0$, for some $C_{\rho}, A, D > 0$, then
the series \eqref{Eqn:v:tExp} converges strongly.
More precisely,
\be
\label{Eqn:AnalCont:Joint}
\SobNorm{\frac{\px^r \pz^t}{(r+t)!} v_m}{2} \leq \underline{K} B^m,
\quad
\forall\ m, r, t \geq 0,
\ee
for some constants $\underline{K}, B > 0$.
\end{Thm}
\begin{proof}
As before, we prove the estimate \eqref{Eqn:AnalCont:Joint} by induction.
Once again, from the definitions of $\phi$ and $\psi$ it is clear that 
$\tau \in C^{\omega}_{\ell}$ for some $C_{\tau} > 0$.
For $m = 0$ we invoke Theorem~\ref{Thm:EllEst:Joint} with 
$F \equiv 0$, $Q = \tau$, and $R \equiv 0$ to realize
\bes
\SobNorm{\frac{\px^r \pz^t}{(r+t)!} v_0}{2}
  \leq \kappa(h) C_{\tau} \frac{A^r}{(r+1)^2} \frac{D^t}{(t+1)^2},
\quad
\forall\ r, t \geq 0,
\ees
and we set $\underline{K} := \kappa(h) C_{\tau}$. We now assume that
\eqref{Eqn:AnalCont:Joint} is true for all $m < M$ and use 
Lemma~\ref{Lemma:Recur:Joint} to invoke Theorem~\ref{Thm:EllEst:Joint}
with
\bes
C_F := \underline{\tilde{C}} \underline{K} B^{M-1},
\quad
C_Q = C_R = 0.
\ees
Therefore we have
\bes
\SobNorm{ \frac{\px^r \pz^t}{(r+t)!}}{2} 
  \leq \kappa(h) \underline{\tilde{C}} \underline{K} B^{M-1} 
  \frac{A^r}{(r+1)^2} \frac{D^t}{(t+1)^2},
  \quad
  \forall\ r, t \geq 0.
\ees
We are finished provided that we choose
\bes
B > \kappa(h) \underline{\tilde{C}}.
\ees
\end{proof}

%
%

\section{Numerical Results}
\label{Sec:NumRes}

At this point we discuss how a numerical implementation of
the HOPE recursions we have derived 
can be utilized to deliver highly accurate approximations of 
scattering returns in a robust and reliable manner. We begin
with a brief description of our High--Order Spectral (HOS)
approach, and then move to comparisons of our approximations
to both the exact solution of layered media scattering
and a smoothed version of this.

%
%

\subsection{Implementation}
\label{Sec:Implementation}

The implementation we describe focuses on approximating the
solution of \eqref{Eqn:LCPRCP:m} in both LCP and RCP cases.
Of course this cannot be achieved for \textit{all} $0 \leq m < \infty$
and we naturally approximate solutions with the finite--order truncation
of \eqref{Eqn:v:Exp},
\be
\label{Eqn:vM:Approx}
v^M(x,z;\delta) := \sum_{m=0}^{M} v_m(x,z) \delta^m,
\ee
where the $v_m(x,z)$ approximately satisfy \eqref{Eqn:LCPRCP:m} 
for $0 \leq m \leq M$.
For this we appeal to the highly accurate family of HOS methods
\cite{GottliebOrszag77,Boyd01,ShenTangWang11}, in particular,
a Fourier--Chebyshev approach. More specifically, we consider
\bes
v_m(x,z) \approx v_m^{N_x,N_z}(x,z) := \sum_{p=-N_x/2}^{N_x/2-1}
  \sum_{q=0}^{N_z} \hat{v}_{m,p,q} T_q(z/h) e^{i \alpha_p x},
\ees
where $T_q$ is the $q$--th Chebyshev polynomial. To prescribe the
Fourier--Chebyshev coefficients, $\{ \hat{v}_{m,p,q} \}$, we follow
the collocation philosophy by demanding that \eqref{Eqn:LCPRCP:m} be
true at the gridpoints
\bes
\{ x_j = j (d/N_x)\ |\ 0 \leq j \leq N_x-1 \},
\quad
\{ z_r = h \cos(\pi r/N_z)\ |\ 0 \leq r \leq N_z \}.
\ees
These constraints generate a sequence of linear systems of equations
(at each wavenumber $p$ and order $m$) which can be rapidly formed
via the fast Fourier and Chebyshev transforms
\cite{GottliebOrszag77,Boyd01,ShenTangWang11}, and quickly solved
using Gaussian elimination \cite{Nicholls19b}.

As with all HOPE algorithms, careful thought is required
for the summation which appears in \eqref{Eqn:vM:Approx}. For instance,
one must form the sum
\bes
S^M(\delta) = S^M_{p,q}(\delta) := \sum_{m=0}^{M} \hat{v}_{m,p,q} \delta^m,
\ees
which approximates the infinite Taylor sum
\bes
S(\delta) = \sum_{m=0}^{\infty} \hat{v}_{m,p,q} \delta^m.
\ees
As justified by Theorem~\ref{Thm:AnalCont}, this Taylor sum, $S(\delta)$,
will converge for all $\delta$ in a (perhaps small) 
neighborhood of the origin in
the complex plane. However, Theorem~\ref{Thm:AnalCont} also concludes
that the domain of analyticity of $S(\delta)$ includes a (perhaps smaller)
neighborhood of the \textit{entire} real axis. A fast and robust algorithm
for accessing this region of extended analyticity is Pad\'e approximation
\cite{BakerGravesMorris96} which we have used with great success in
the past \cite{NichollsReitich00b,Nicholls19b,NichollsVo23,NichollsVo24}.
The Pad\'e approximant of $S^M(\delta)$ is the rational function
\bes
[P/Q](\delta) := \frac{a^P(\delta)}{b^Q(\delta)}
  = \frac{\sum_{p=0}^{P} a_p \delta^p}
  {\sum_{q=0}^{Q} b_q \delta^q},
\ees
where
\bes
[P/Q](\delta) = S^M(\delta) + \BigOh{\delta^{P+Q+1}}.
\ees
Well--known formulas exist for the recovery of these 
$\{ a_p, b_q \}$ \cite{BakerGravesMorris96} and the
resulting approximants often deliver remarkably accurate 
approximations of $S(\delta)$, even well outside of the \textit{disk}
of convergence of the original Taylor series (see, e.g., Chapter~8
of \cite{BenderOrszag78}).

%
%

\subsection{Scattering by Triply Layered Media}
\label{Sec:LayMedia}

In order to demonstrate the utility of our numerical implementation
of the HOPE recursions, we consider a specific problem featuring
an exact solution: Scattering of plane electromagnetic radiation
by a layered medium featuring achiral and chiral materials with
flat interfaces. In more detail, we consider a material with
constant permittivity, $\epsilon_0$, and permeability, $\mu_0$,
and layered chirality
\be
\label{Eqn:Chi:Layered}
\chi = \chi(z) = \begin{cases} 
  0, & t < z < h, \\
  \chi', & -t < z < t, \\
  0, & -h < z < -t,
  \end{cases}
\ee
for $0 < t < h$, with incident radiation $\tau$, \eqref{Eqn:tau}.
It is not difficult to show that the solution of this problem is
\be
\label{Eqn:Flat:Exact}
v(x,z) = \begin{cases} 
  R \exp(i \alpha x + i \gamma z), & t < z < h, \\
  U \exp(i \alpha x + i \gamma' z)
  + D \exp(i \alpha x - i \gamma' z), & -t < z < t, \\
  T \exp(i \alpha x - i \gamma z), & -h < z < -t,
  \end{cases}
\ee
where
\bes
\alpha^2 + \gamma^2 = \kz^2,
\quad
\alpha^2 + (\gamma')^2 = \begin{cases} \kz^2/(1-\kz \chi')^2, & \text{LCP}, \\
  \kz^2/(1+\kz \chi')^2, & \text{RCP}, \end{cases}
\ees
and jump conditions at the interfaces deliver a \textit{linear}
system of four equations for the four unknowns, $\{ R, U, D, T \}$
\cite{Yeh05}. These can be readily solved
by Gaussian elimination, and we use these as a convenient \textit{exact}
solution against which to test our implementation.

Our theorems in their current form do not permit such an irregular
chirality function (merely $L^{\infty}$), but, while such a theory
could perhaps be discovered, the resulting numerics would certainly
be of very low quality. To avoid this in our numerics
we approximate the layered
chirality \eqref{Eqn:Chi:Layered} by the smoothed one defined
in \eqref{Eqn:X:Smooth} (see Figure~\ref{Fig:Chirality}) with the
choice $w=100$. With this in hand we can be more precise about
our experiments: We selected the following particular geometric
constants
\bes
d = 0.8,
\quad
h = 0.95,
\quad
t = 0.5,
\ees
and electromagnetic constants
\bes
\lambda = 0.7,
\quad
\theta = 30^{\circ}.
\ees
Akin to \cite{Nicholls19b} we chose two configurations: (i.) small 
($\chi' = 0.01$) and (ii.) large ($\chi' = 0.04$) deviations from
the zero background chirality. For the small deviation we chose
\bes
16 \leq N_x \leq 32,
\quad
2 \leq M \leq 16,
\ees
and for the large perturbation case
\bes
16 \leq N_x \leq 64,
\quad
2 \leq M \leq 32,
\ees
where $N_z = N_x$. We measured errors in the $L^{\infty}$ norm as
\bes
\text{Error} = \text{Error}^{N_x,N_z,M}
  := \SupNorm{v^{N_x,N_z,M}-v^{\text{Exact}}}.
\ees

%
%

\subsection{Convergence}
\label{Sec:Conv}

We now present the results of our numerical experiments comparing
numerical solution of the HOPE recursions versus the flat--interface
exact solution. We begin with the small deviation case where
$\chi'=0.01$ and show the convergence of our scheme as
$N_x=N_z$ (Figures~\ref{Fig:ErrNx:Exact:Small:LCP}
and \ref{Fig:ErrNx:Exact:Small:RCP}) and 
$M$ (Figures~\ref{Fig:ErrM:Exact:Small:LCP}
and \ref{Fig:ErrM:Exact:Small:RCP}) are refined. In a manner
similar to our previous work \cite{Nicholls19b}, the results are
satisfying in that the error appears to decay as the appropriate
parameters are increased, however, they are also disappointing
as the \textit{spectral} nature of our HOPE method with respect
to both the spatial ($x$ and $z$) and deformation ($\delta$)
parameters would suggest an \textit{exponential} rate of
convergence \cite{GottliebOrszag77,Boyd01,ShenTangWang11}.
However, this mystery is quickly resolved by noting that solutions
in both LCP and RCP will be merely $H^1$ in the case where
$\chiv \in L^{\infty}$. In the next section we return to this issue.
%
%
%
%
%
\begin{figure}[hbt]
  \begin{center}
  \includegraphics[width=0.8\textwidth]{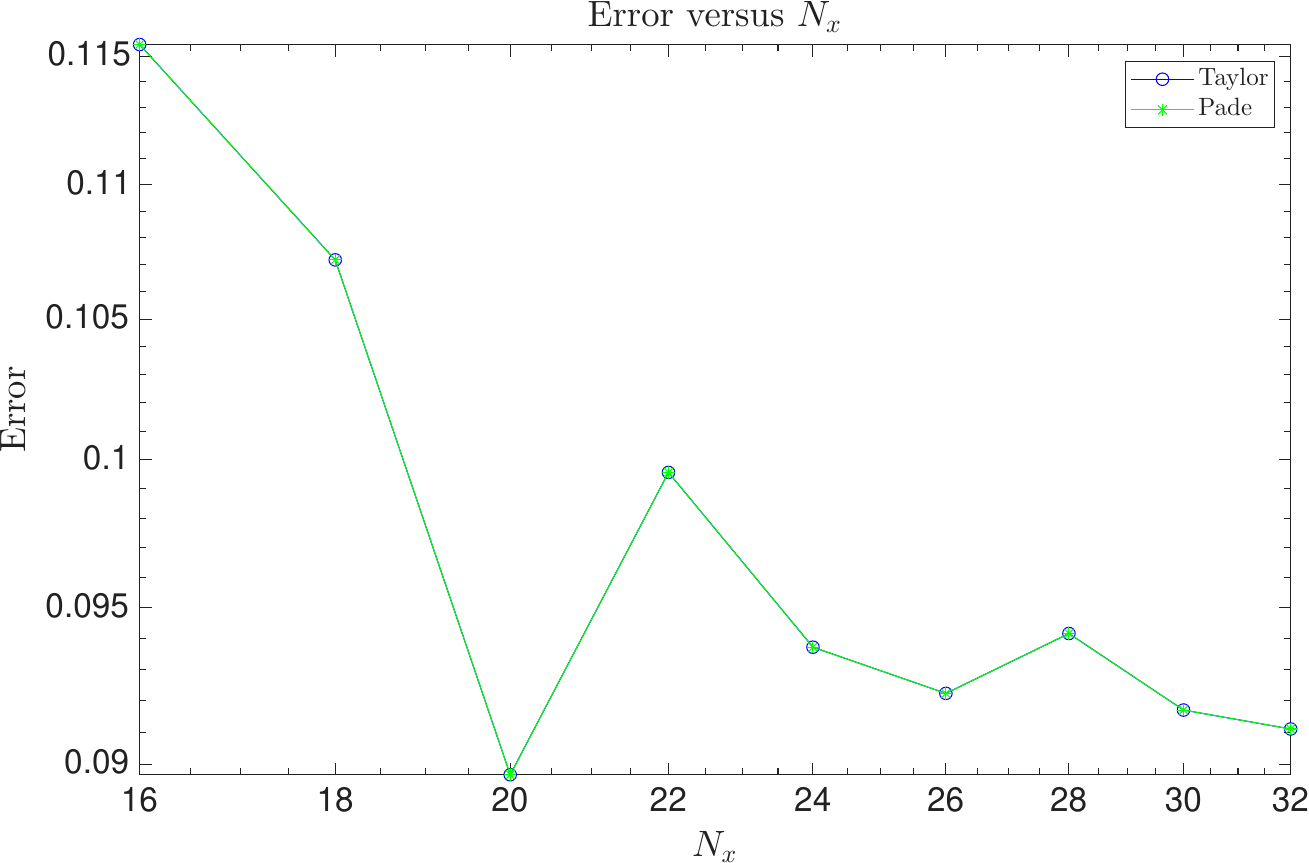}
  \caption{Error versus $N_x$ for $\chi' = 0.01$ compared
  with the exact solution: LCP.}
  \label{Fig:ErrNx:Exact:Small:LCP}
  \end{center}
\end{figure}
\begin{figure}[hbt]
  \begin{center}
  \includegraphics[width=0.8\textwidth]{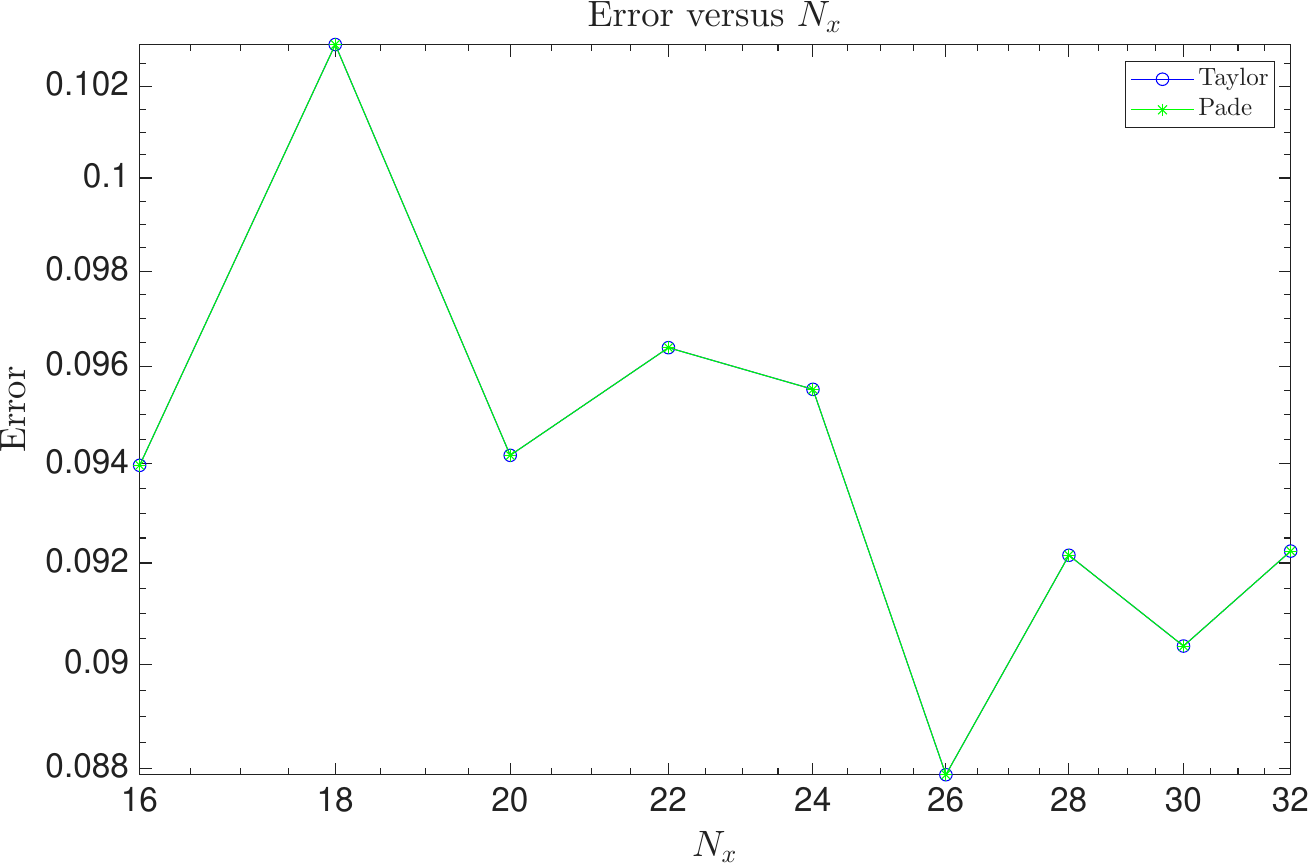}
  \caption{Error versus $N_x$ for $\chi' = 0.01$ compared
  with the exact solution: RCP.}
  \label{Fig:ErrNx:Exact:Small:RCP}
  \end{center}
\end{figure}
%
%
%
%
%
\begin{figure}[hbt]
  \begin{center}
  \includegraphics[width=0.8\textwidth]{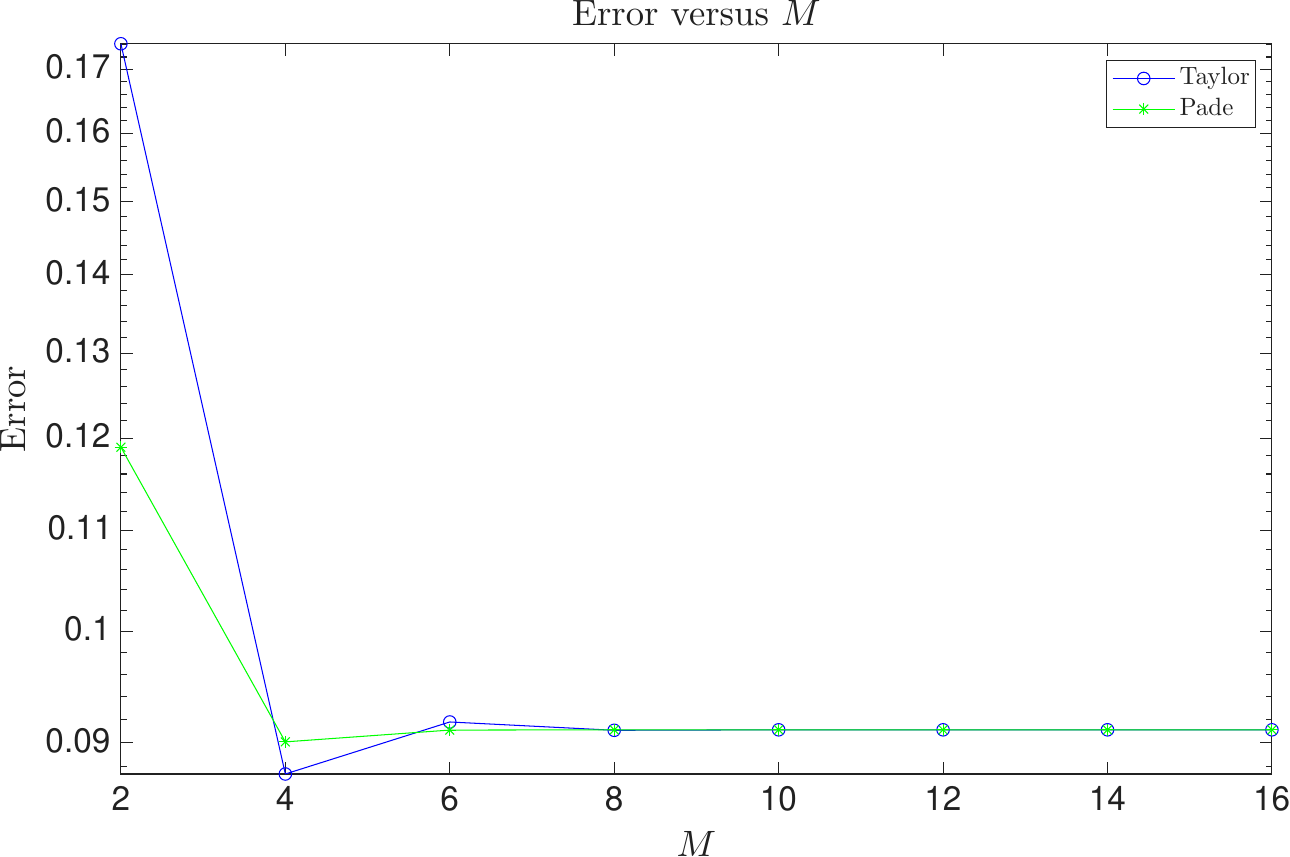}
  \caption{Error versus $M$ for $\chi' = 0.01$ compared
  with the exact solution: LCP.}
  \label{Fig:ErrM:Exact:Small:LCP}
  \end{center}
\end{figure}
\begin{figure}[hbt]
  \begin{center}
  \includegraphics[width=0.8\textwidth]{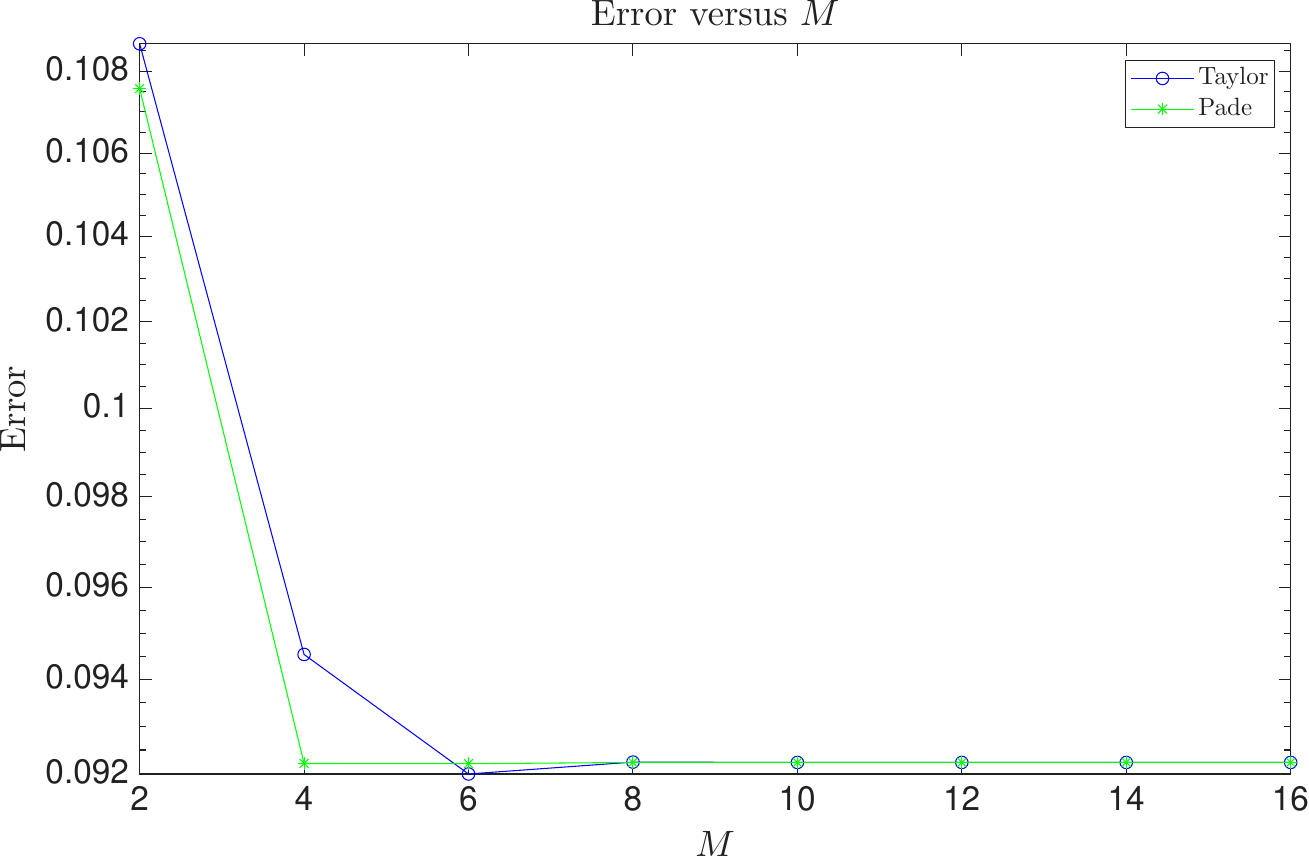}
  \caption{Error versus $M$ for $\chi' = 0.01$ compared
  with the exact solution: RCP.}
  \label{Fig:ErrM:Exact:Small:RCP}
  \end{center}
\end{figure}

We repeated the above experiments in the case of a large deviation
$\chi'=0.04$ and report the results of this in 
Figures~\ref{Fig:ErrNx:Exact:Large:LCP},
\ref{Fig:ErrNx:Exact:Large:RCP},
\ref{Fig:ErrM:Exact:Large:LCP},
and \ref{Fig:ErrM:Exact:Large:RCP}
for $N_x=N_z$ and $M$ refinement, respectively. Here we notice a
new phenomena: Substantial divergence of behavior based upon
summation method. More specifically, while Pad\'e summation
delivers reasonable (though not outstanding) results, Taylor
summation produces completely unreliable answers. The reason
for this is that for this choice of $\chi'$, the value $\delta=-1$
is \textit{outside} the disk of convergence of our solution
\eqref{Eqn:v:Exp} and only with numerical analytic continuation
can we recover useful answers. However, even in the
Pad\'e case, we are disappointed with the \textit{rate} of refinement
which, as before, is readily explained by the lack of smoothness
in the underlying exact solution.
%
%
%
%
%
\begin{figure}[hbt]
  \begin{center}
  \includegraphics[width=0.8\textwidth]{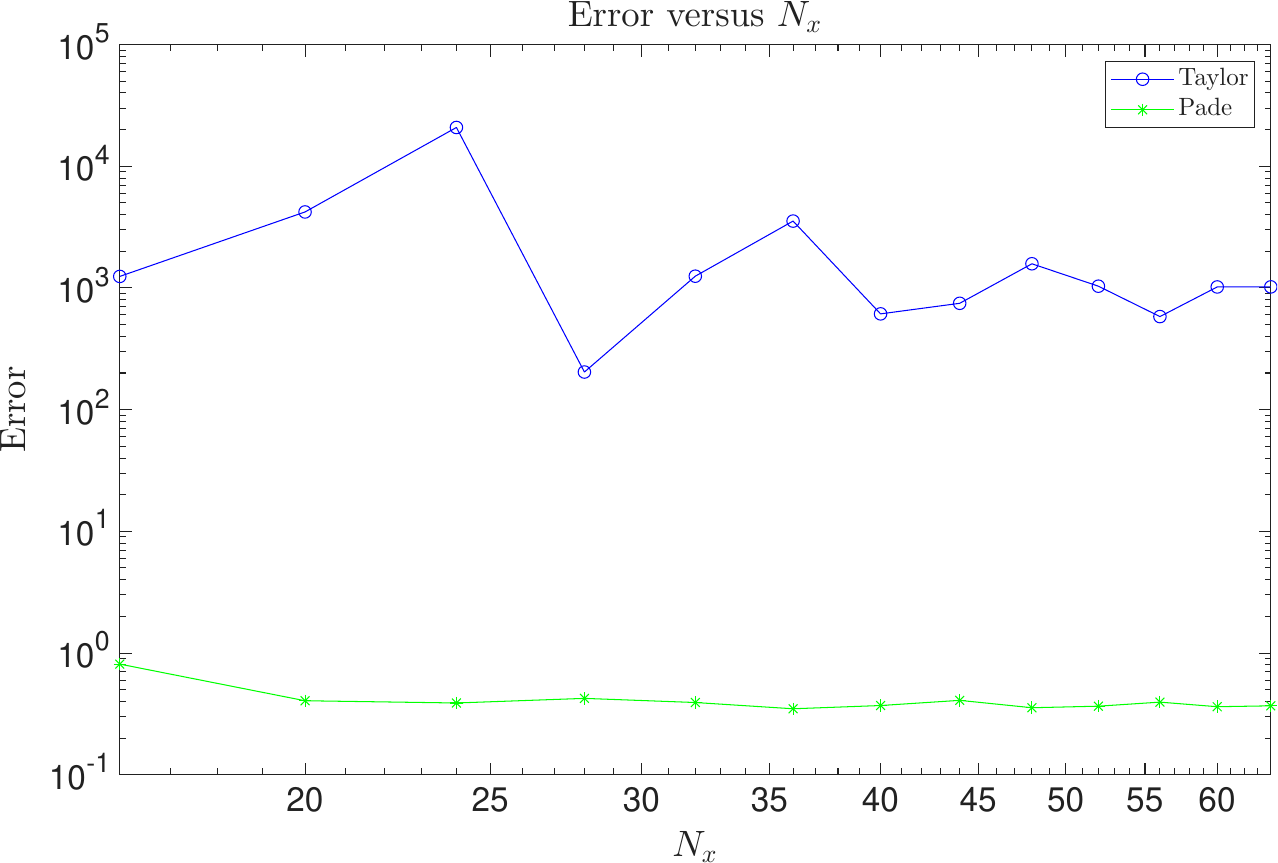}
  \caption{Error versus $N_x$ for $\chi' = 0.04$ compared
  with the exact solution: LCP.}
  \label{Fig:ErrNx:Exact:Large:LCP}
  \end{center}
\end{figure}
\begin{figure}[hbt]
  \begin{center}
  \includegraphics[width=0.8\textwidth]{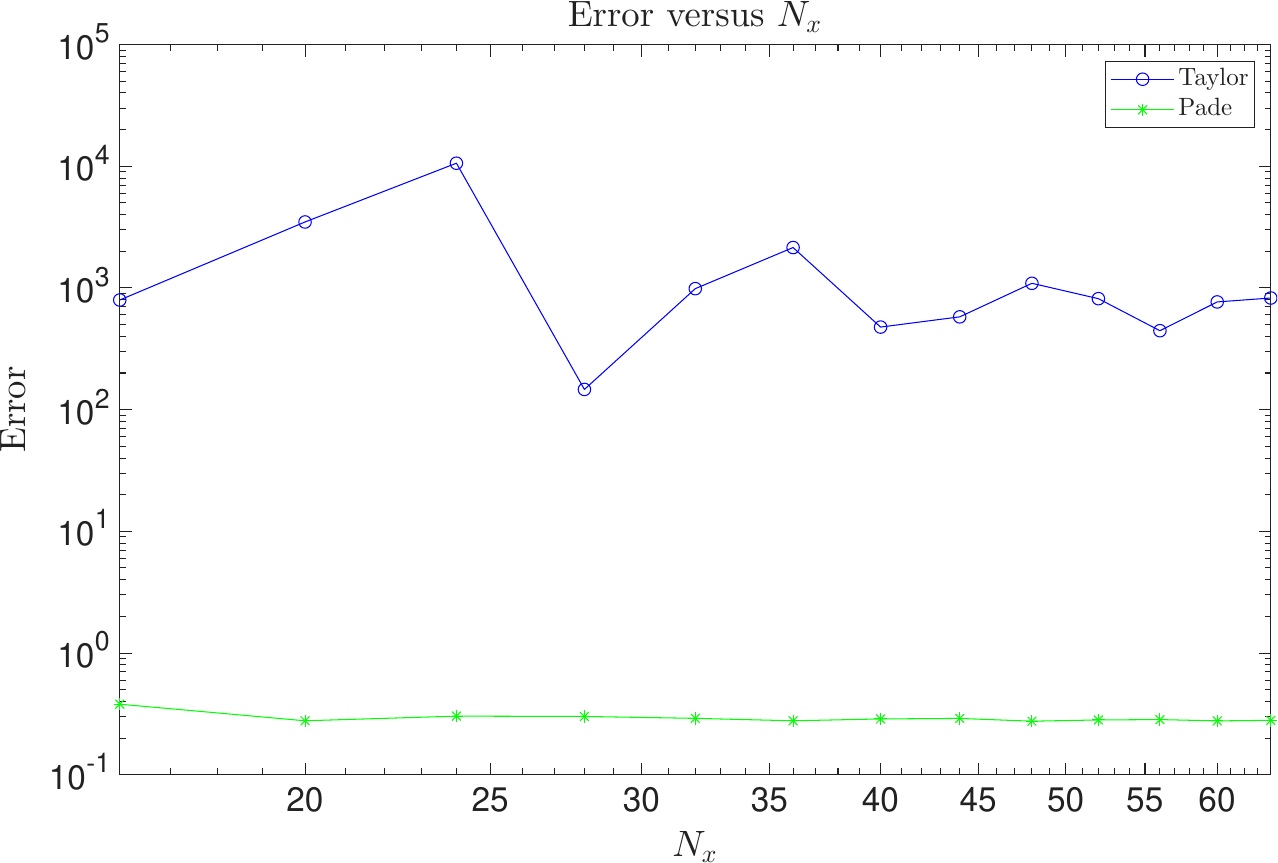}
  \caption{Error versus $N_x$ for $\chi' = 0.04$ compared
  with the exact solution: RCP.}
  \label{Fig:ErrNx:Exact:Large:RCP}
  \end{center}
\end{figure}
%
%
%
%
%
\begin{figure}[hbt]
  \begin{center}
  \includegraphics[width=0.8\textwidth]{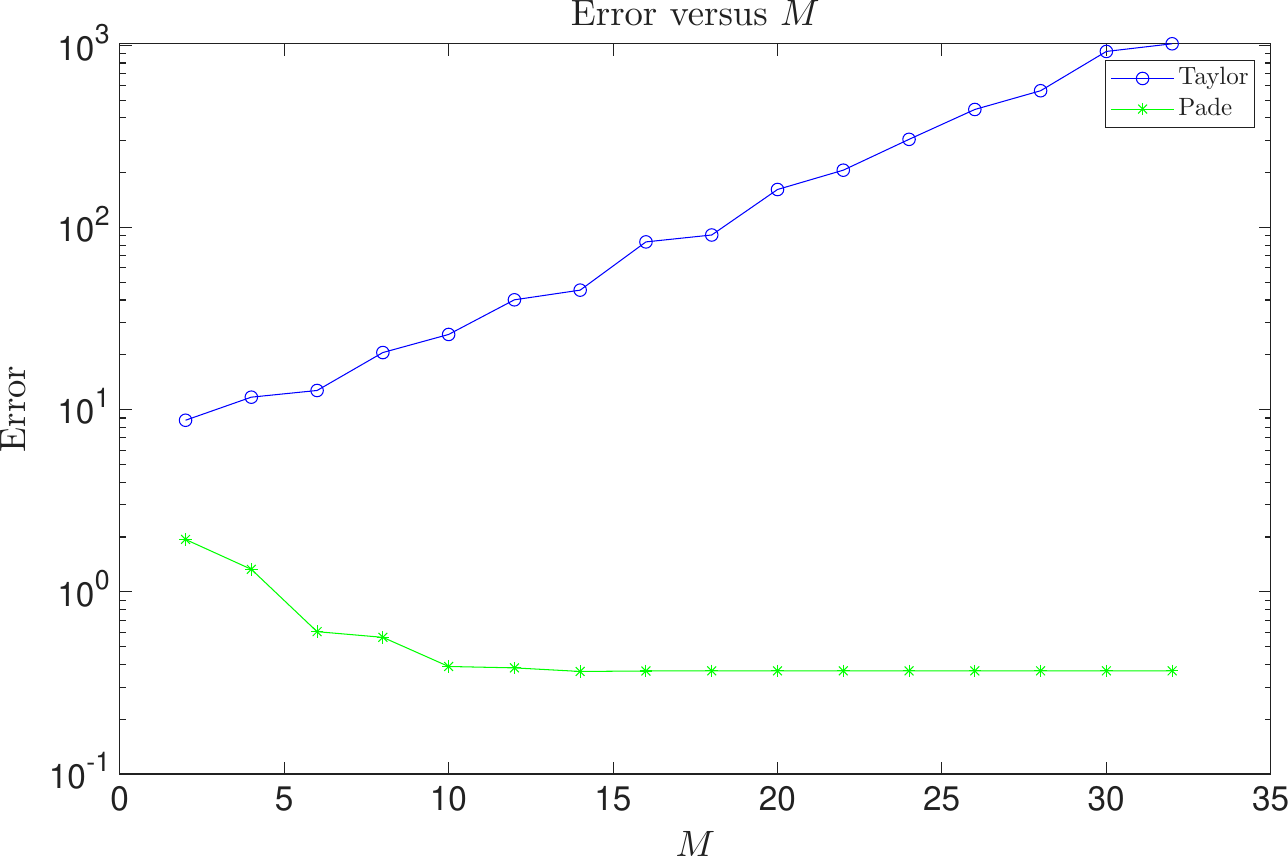}
  \caption{Error versus $M$ for $\chi' = 0.04$ compared
  with the exact solution: LCP.}
  \label{Fig:ErrM:Exact:Large:LCP}
  \end{center}
\end{figure}
\begin{figure}[hbt]
  \begin{center}
  \includegraphics[width=0.8\textwidth]{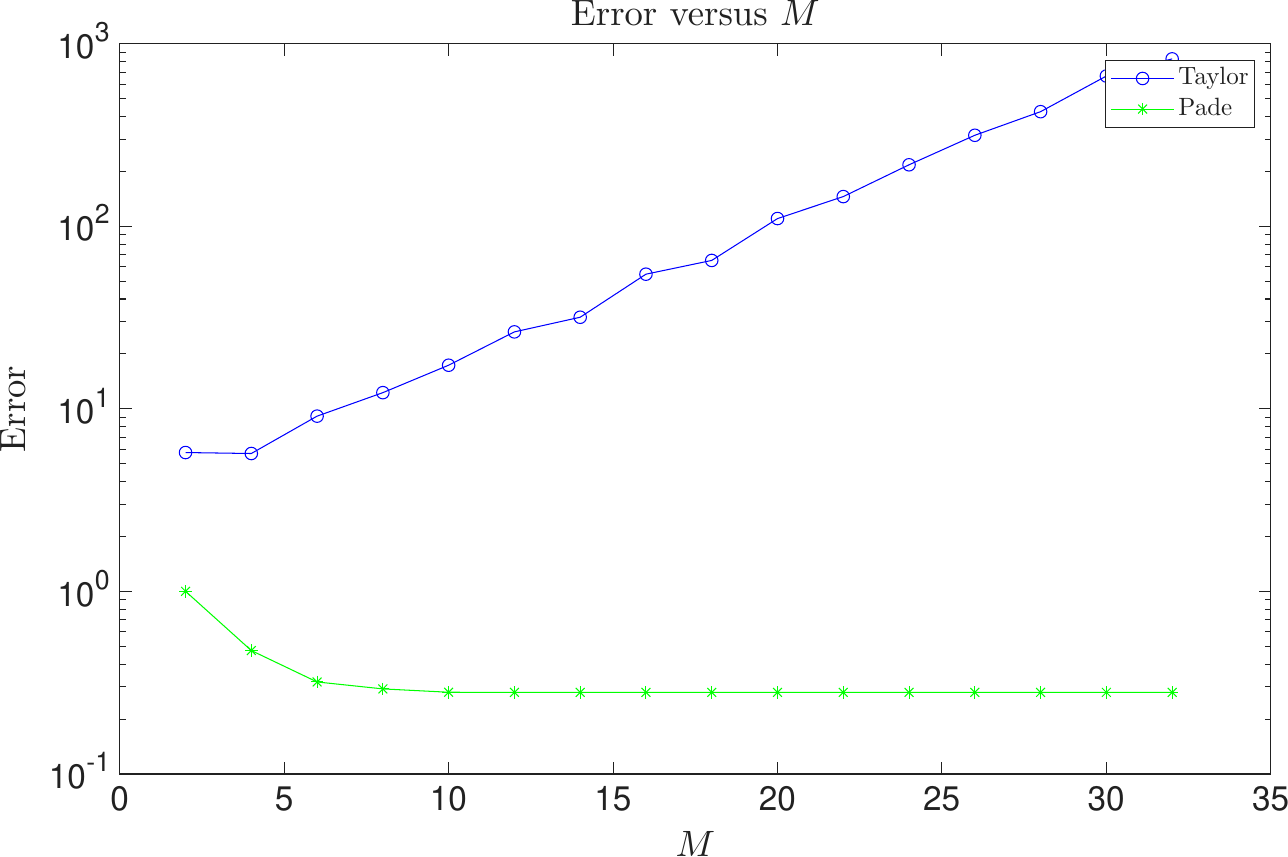}
  \caption{Error versus $M$ for $\chi' = 0.04$ compared
  with the exact solution: RCP.}
  \label{Fig:ErrM:Exact:Large:RCP}
  \end{center}
\end{figure}
%

%
%

\subsection{Convergence to Smoothed Solution}
\label{Sec:ConvSmooth}

As we mentioned above, our results thus far have been somewhat
disappointing in light of the spectral rate of convergence we
were anticipating. To address this concern we repeated these
experiments where the exact solution is replaced by a
``smoothed'' one.  For this we took advantage of the fact
that if $\rho$ is independent of $x$ then the solution of
\eqref{Eqn:LCPRCP} can be expanded in a (generalized)
Fourier series as
\bes
v(x,z) = \sump \hat{v}_p(z) e^{i \alpha_p x},
\ees
where the resulting $\hat{v}_p(z)$ satisfy
\bse
\label{Eqn:TwoPointBVP}
\begin{align}
& \rho(z) \pz \left[ \rho(z) \pz \hat{v}_p(z) \right]
  + (\kz^2 - \rho(z)^2 \alpha_p^2) \hat{v}_p(z) = 0,
  && -h < z < h, \\
& -\pz \hat{v}_p - (-i \gammau_p) \hat{v}_p = \hat{\tau}_p, 
  && z = h, \\
& \pz \hat{v}_p - (-i \gammaw_p) \hat{v}_p = 0,
  && z = -h,
\end{align}
\ese
and
\bes
\hat{\tau}_p = \frac{1}{d} \int_0^d \tau(x) e^{-i \alpha_p x} \dx,
\ees
is the $p$--th Fourier coefficient of $\tau(x)$ \cite{Yeh05}.
If $\rho(z)$ is piecewise constant then one can express the
solution in each homogeneous layer as upward-- and 
downward--propagating modes, c.f. \eqref{Eqn:Flat:Exact},
and jump conditions at the interfaces give rise to the readily
solved Fresnel equations. To
avoid the low regularity of the resulting solutions, we
followed our HOPE approach and \textit{approximated} the piecewise
constant chirality by the smoothed form \eqref{Eqn:X:Smooth}.
While this yields a much
smoother approximate solution, one must utilize a numerical
algorithm to solve \eqref{Eqn:TwoPointBVP}. For this
we selected the Chebyshev method 
\cite{GottliebOrszag77,Boyd01,ShenTangWang11}
using the same number of Chebyshev coefficients, specified
by $N_z$, as the relevant HOPE simulation.

We revisited our earlier simulations and report in 
Figures~\ref{Fig:ErrNx:Smooth:Small:LCP},
\ref{Fig:ErrNx:Smooth:Small:RCP},
\ref{Fig:ErrM:Smooth:Small:LCP},
and 
\ref{Fig:ErrM:Smooth:Small:RCP} results of $N_x=N_z$ and
$M$ refinement, respectively, in the small deviation
case $\chi' = 0.01$. Now we notice on these log--linear
plots the exponential rates of convergence we were hoping
to observe. We also notice an enhanced rate of convergence,
with Pad\'e approximation,
even within the disk of convergence which has been
observed in the past.
%
%
%
%
%
\begin{figure}[hbt]
  \begin{center}
  \includegraphics[width=0.8\textwidth]{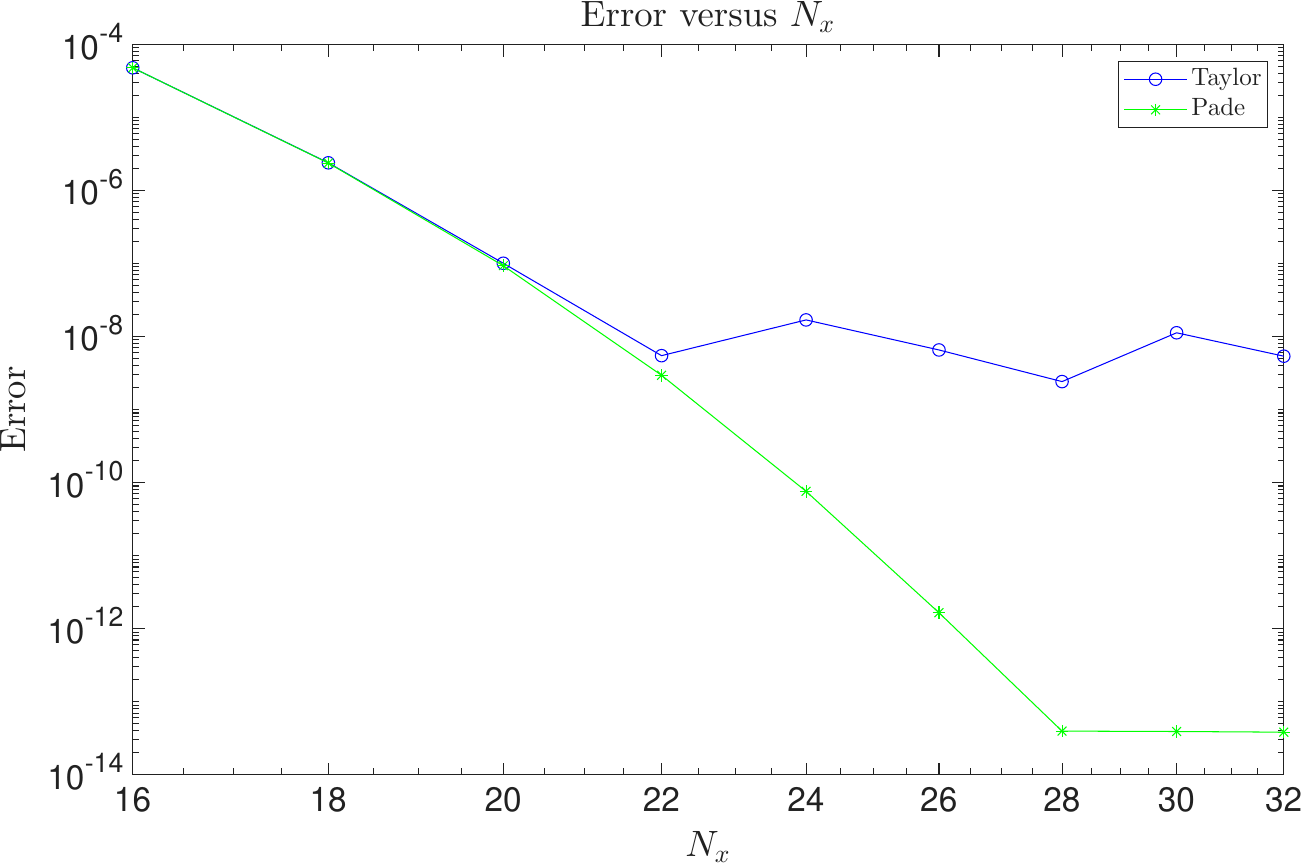}
  \caption{Error versus $N_x$ for $\chi' = 0.01$ compared
  with the smoothed exact solution: LCP.}
  \label{Fig:ErrNx:Smooth:Small:LCP}
  \end{center}
\end{figure}
\begin{figure}[hbt]
  \begin{center}
  \includegraphics[width=0.8\textwidth]{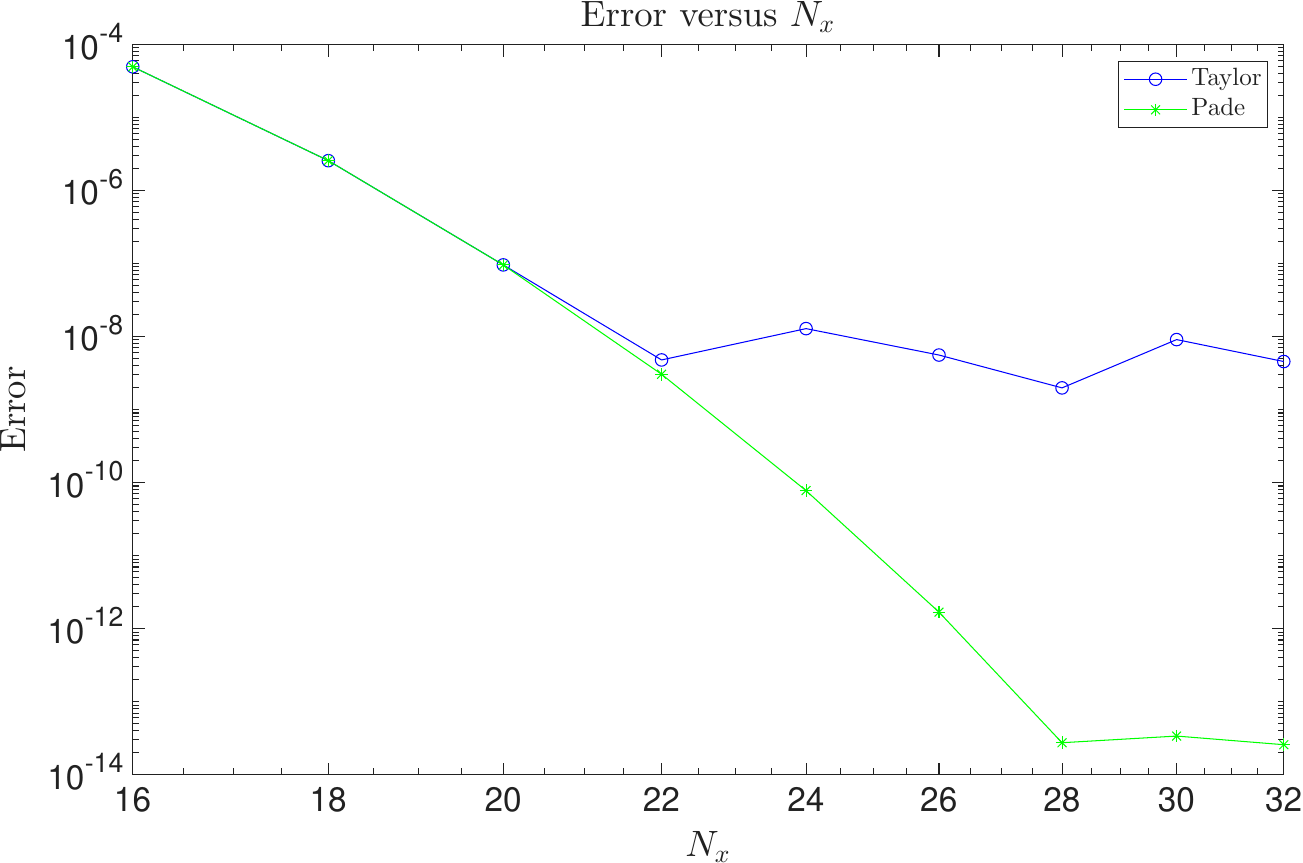}
  \caption{Error versus $N_x$ for $\chi' = 0.01$ compared
  with the smoothed exact solution: RCP.}
  \label{Fig:ErrNx:Smooth:Small:RCP}
  \end{center}
\end{figure}
%
%
%
%
%
\begin{figure}[hbt]
  \begin{center}
  \includegraphics[width=0.8\textwidth]{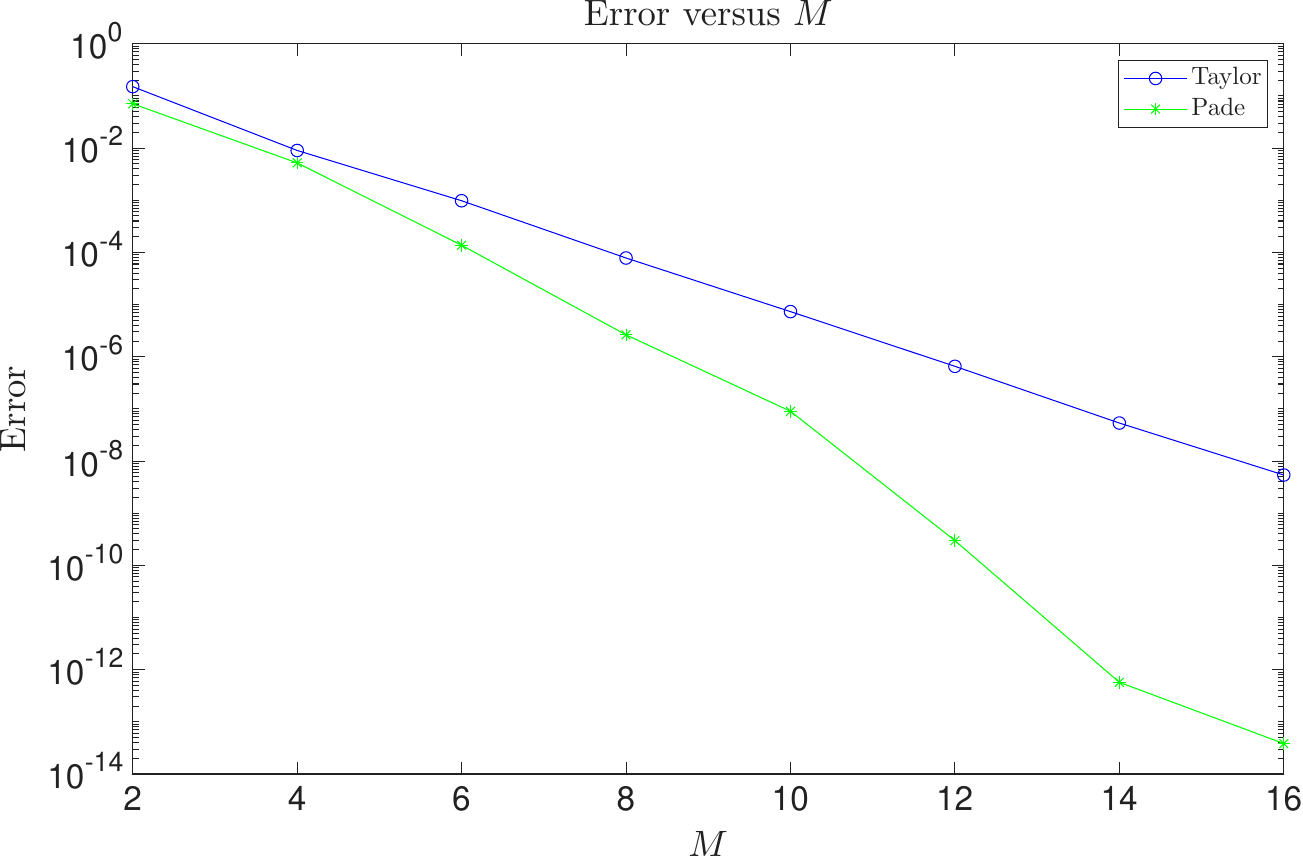}
  \caption{Error versus $M$ for $\chi' = 0.01$ compared
  with the smoothed exact solution: LCP.}
  \label{Fig:ErrM:Smooth:Small:LCP}
  \end{center}
\end{figure}
\begin{figure}[hbt]
  \begin{center}
  \includegraphics[width=0.8\textwidth]{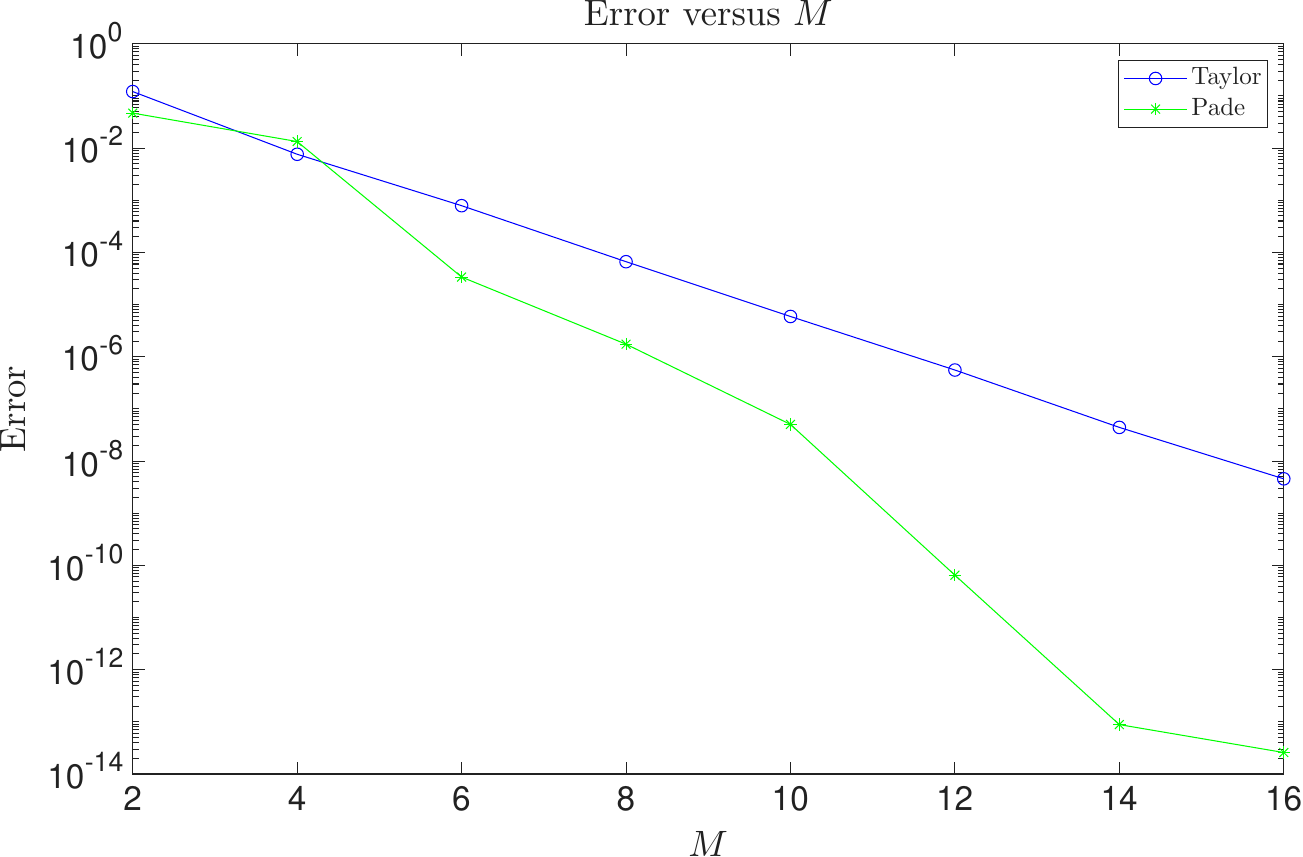}
  \caption{Error versus $M$ for $\chi' = 0.01$ compared
  with the smoothed exact solution: RCP.}
  \label{Fig:ErrM:Smooth:Small:RCP}
  \end{center}
\end{figure}

We close with the large deviation case, $\chi' = 0.04$,
and plot the outcomes of our simulations in
Figures~\ref{Fig:ErrNx:Smooth:Large:LCP},
\ref{Fig:ErrNx:Smooth:Large:RCP},
\ref{Fig:ErrM:Smooth:Large:LCP},
and
\ref{Fig:ErrM:Smooth:Large:RCP} for increasing $N_x=N_z$
and $M$, respectively. As $\chi' = 0.04$ pushes the
value $\delta=-1$ outside the disk of analyticity of
\eqref{Eqn:v:Exp} we cannot expect reliable results
from Taylor summation (which we observe), however,
Pad\'e summation delivers highly resolved answers
in a rapid and robust fashion.
%
%
%
%
%
\begin{figure}[hbt]
  \begin{center}
  \includegraphics[width=0.8\textwidth]{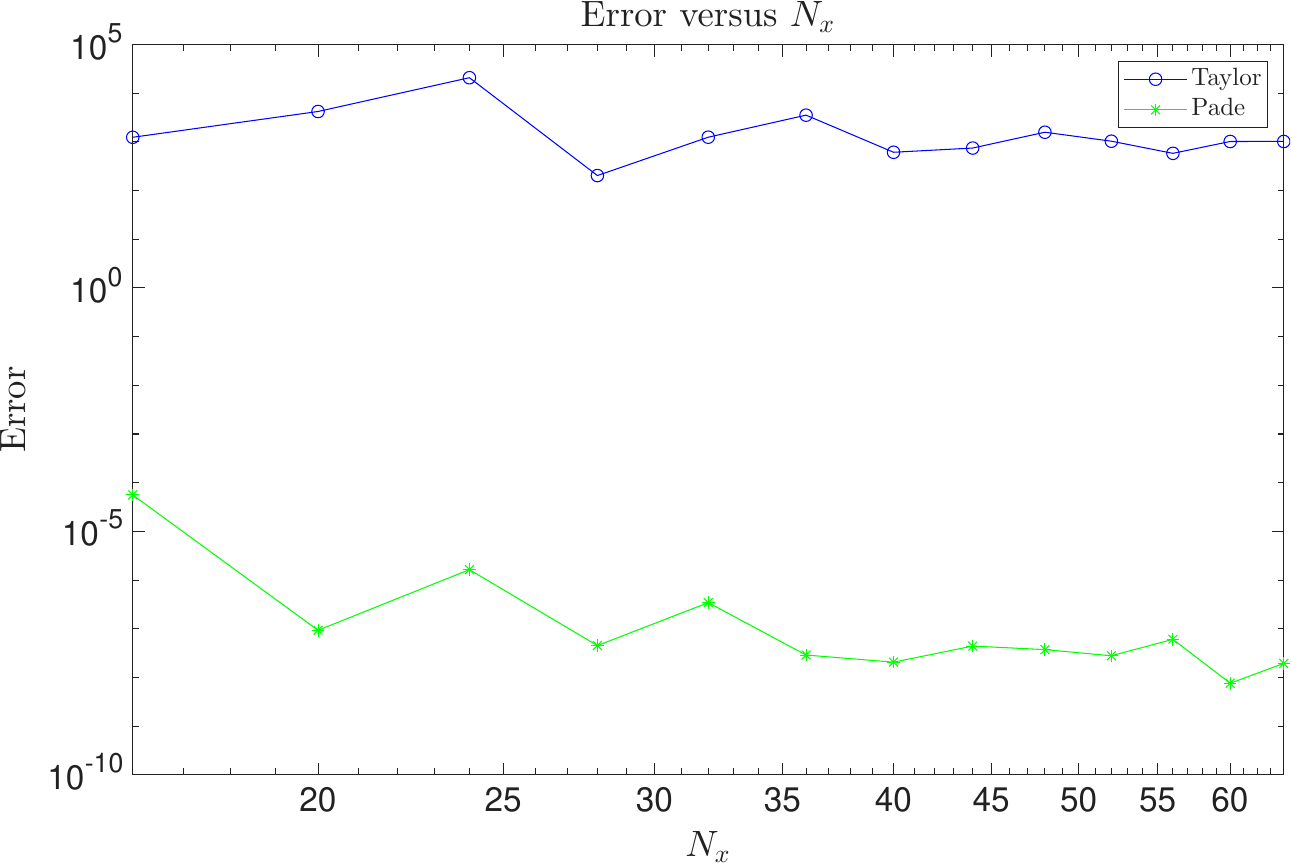}
  \caption{Error versus $N_x$ for $\chi' = 0.04$ compared
  with the smoothed exact solution: LCP.}
  \label{Fig:ErrNx:Smooth:Large:LCP}
  \end{center}
\end{figure}
\begin{figure}[hbt]
  \begin{center}
  \includegraphics[width=0.8\textwidth]{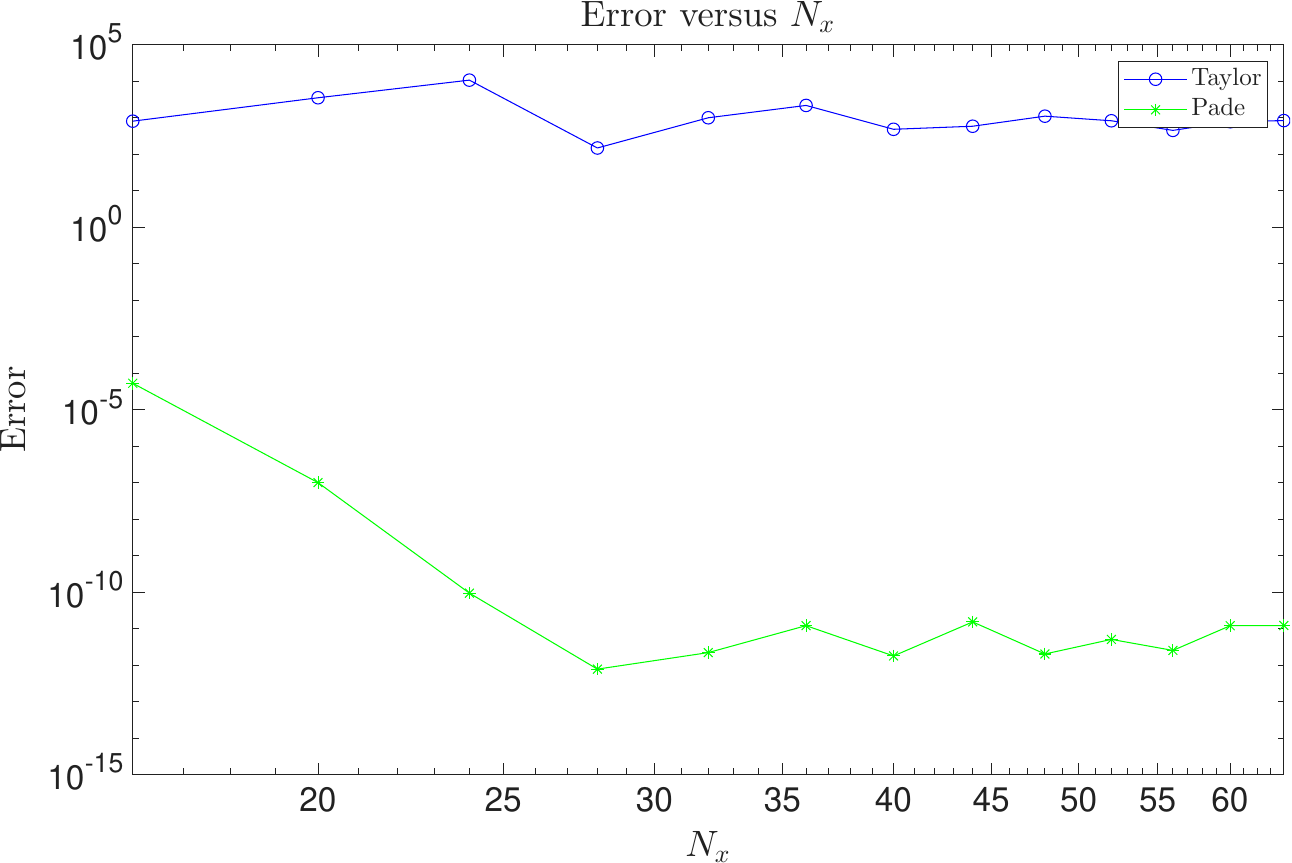}
  \caption{Error versus $N_x$ for $\chi' = 0.04$ compared
  with the smoothed exact solution: RCP.}
  \label{Fig:ErrNx:Smooth:Large:RCP}
  \end{center}
\end{figure}
%
%
%
%
%
\begin{figure}[hbt]
  \begin{center}
  \includegraphics[width=0.8\textwidth]{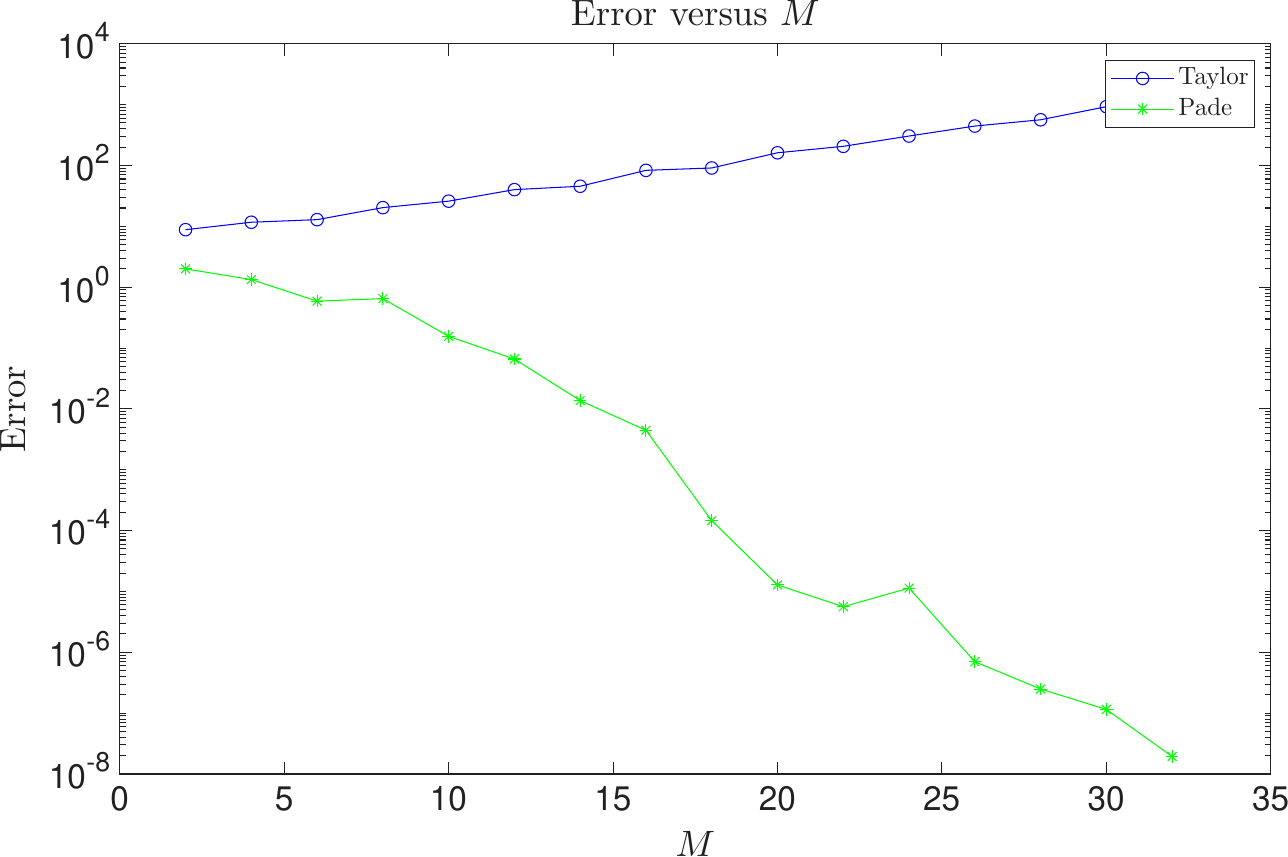}
  \caption{Error versus $M$ for $\chi' = 0.04$ compared
  with the smoothed exact solution: LCP.}
  \label{Fig:ErrM:Smooth:Large:LCP}
  \end{center}
\end{figure}
\begin{figure}[hbt]
  \begin{center}
  \includegraphics[width=0.8\textwidth]{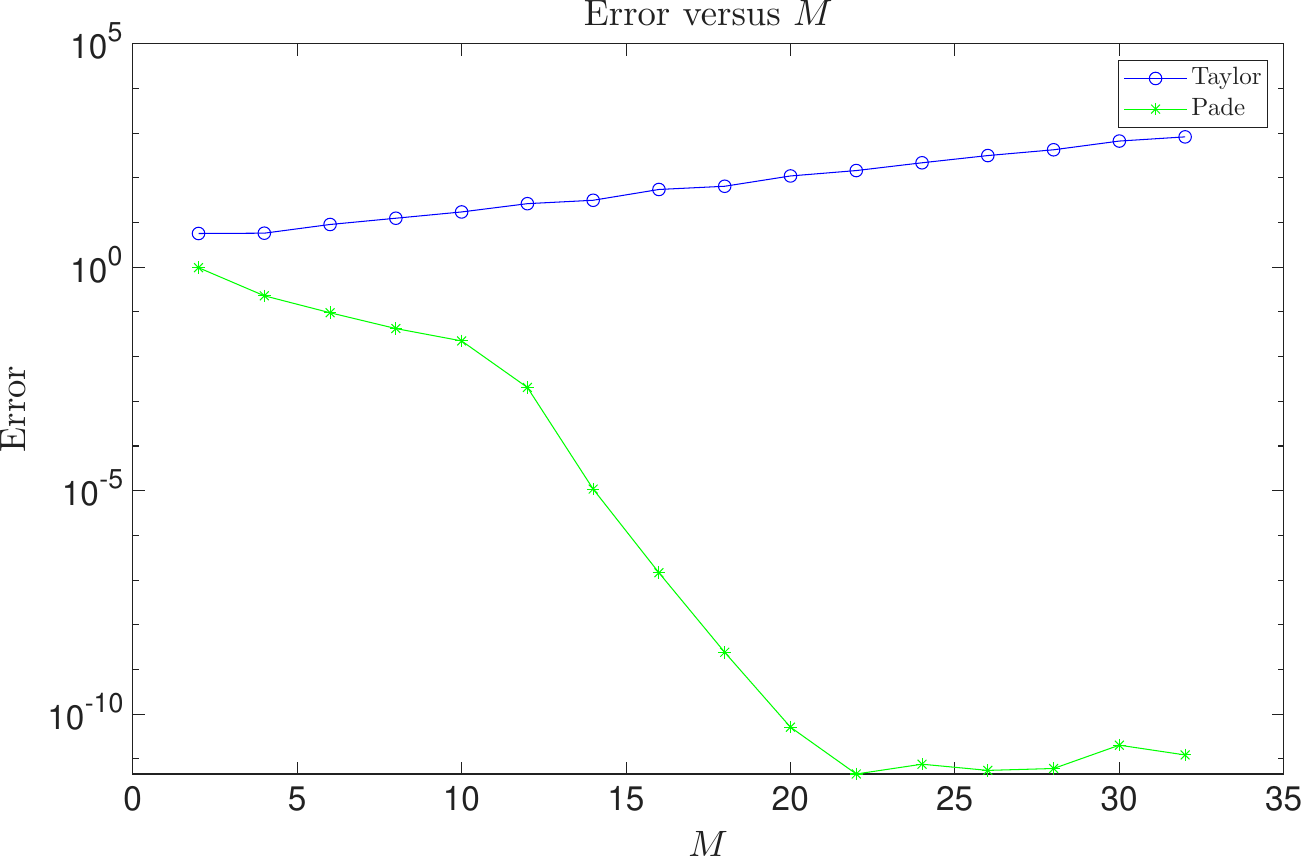}
  \caption{Error versus $M$ for $\chi' = 0.04$ compared
  with the smoothed exact solution: RCP.}
  \label{Fig:ErrM:Smooth:Large:RCP}
  \end{center}
\end{figure}
%

%
%

\section{Conclusion}
\label{Sec:Conc}

In this work we have considered the interaction of electromagnetic
waves with a periodic chiral structure. We focused on the special
case of a uniform permeability/permeability, $y$--invariant 
structure with radiation aligned with this axis. In this case
the governing equations can be decoupled into Left-- and
Right--Circular polarizations (LCP/RCP) for \textit{scalar} fields,
and we have proposed a High--Order Perturbation of Envelopes (HOPE)
approach for their solution. We established two novel results:
First that these LCP/RCP fields are not only analytic in small
envelope deformation, but also have a domain of analyticity
which includes the entire real axis. Additionally, we showed
that the fields depend \textit{jointly} analytically upon
the deviation parameters and the spatial variables. Beyond
this we demonstrated how a numerical implementation of this
HOPE algorithm can be used to approximate solutions in a rapid,
robust, and reliable manner.

%
%

\section*{Acknowledgments}

D.P.N.\ gratefully acknowledges support from the National Science
Foundation through grant
No.~DMS--2111283.

%
%


%
%

\bibliography{nicholls}

@article{NichollsReitich99,
  fauthor = {David P. Nicholls and Fernando Reitich},
  author = {D. P. Nicholls and F. Reitich},
  title = {A new approach to analyticity of {D}irichlet-{N}eumann
           operators},
  journal = {Proc. Roy. Soc. Edinburgh Sect. A},
  fjournal = {Proceedings of the Royal Society of Edinburgh. Section A.
              Mathematics},
  year = {2001},
  volume = {131},
  number = {6},
  pages = {1411--1433} }

@article{NichollsReitich00b,
  fauthor = {David P. Nicholls and Fernando Reitich},
  author = {D. P. Nicholls and F. Reitich},
  title = {Analytic Continuation of {D}irichlet-{N}eumann Operators},
  journal = {Numer. Math.},
  fjournal = {Numerische Mathematik},
  year = {2003},
  volume = {94},
  number = {1},
  pages = {107--146}}

@article{NichollsTaber06,
  fauthor = {David P. Nicholls and Mark Taber},
  author = {D. P. Nicholls and M. Taber},
  title = {Joint Analyticity and Analytic Continuation for 
           {D}irichlet--{N}eumann Operators on Doubly Perturbed Domains},
  journal = {J. Math. Fluid Mech.},
  fjournal = {Journal of Mathematical Fluid Mechanics},
  volume = {10},
  number = {2},
  pages = {238--271},
  year = {2008} }

@article{Nicholls19b,
  author = {D. P. Nicholls},
  title = {A High--Order Perturbation of Envelopes ({HOPE}) Method for
    Scattering by Periodic Inhomogeneous Media},
  journal = {Quarterly of Applied Mathematics},
  volume = {78},
  pages = {725-757},
  year = {2020}
}

@article{NichollsVo24,
  fauthor = {David P. Nicholls and Liet Vo},
  author = {D. P. Nicholls and L. Vo},
  title = {A High-Order Perturbation of Envelopes ({HOPE}) Method for 
    Vector Electromagnetic Scattering by Periodic Inhomogeneous Media:
    Analytic Continuation},
  journal = {Journal of Differential Equations},
  volume = {422},
  pages = {106--151},
  year = {2025} }

@article{NichollsVo23,
  fauthor = {David P. Nicholls and Liet Vo},
  author = {D. P. Nicholls and L. Vo},
  title = {A High-Order Perturbation of Envelopes ({HOPE}) Method for 
    Vector Electromagnetic Scattering by Periodic Inhomogeneous Media:
    Joint Analyticity},
  journal = {SIAM Journal on Applied Mathematics},
  volume = {85},
  issue = {2},
  pages = {755--778},
  year = {2025} }

@book {ShenTangWang11,
    AUTHOR = {Shen, Jie and Tang, Tao and Wang, Li-Lian},
     TITLE = {Spectral methods},
    SERIES = {Springer Series in Computational Mathematics},
    VOLUME = {41},
      NOTE = {Algorithms, analysis and applications},
 PUBLISHER = {Springer, Heidelberg},
      YEAR = {2011},
     PAGES = {xvi+470}
}

@book {Petit80,
     TITLE = {Electromagnetic theory of gratings},
    FEDITOR = {Petit, Roger},
    EDITOR = {Petit, R.},
 PUBLISHER = {Springer-Verlag},
   ADDRESS = {Berlin},
      YEAR = {1980},
     PAGES = {xv+284}
}

@book{TaroudakisMakrakis01,
   AUTHOR = {Taroudakis, Michael and Makrakis, George},
    TITLE = {Inverse problems in underwater acoustics},
PUBLISHER = {Springer-Verlag},
  ADDRESS = {New York},
     YEAR = {2001},
    PAGES = {216}
}

@phdthesis{ArensHab,
  Type = {Habilitationsschrift},
  Title = {Scattering by Biperiodic Layered Media:
    The Integral Equation Approach},
  Author = {Tilo Arens},
  School = {Karlsruhe Institute of Technology},
  Year = {2009},
}

@article {AmmariBao08,
    AUTHOR = {Ammari, Habib and Bao, Gang},
     TITLE = {Coupling of finite element and boundary element methods for
              the scattering by periodic chiral structures},
   JOURNAL = {J. Comput. Math.},
  FJOURNAL = {Journal of Computational Mathematics},
    VOLUME = {26},
      YEAR = {2008},
    NUMBER = {3},
     PAGES = {261--283}
}

@article {AmmariBao03,
    AUTHOR = {Ammari, Habib and Bao, Gang},
     TITLE = {Maxwell's equations in periodic chiral structures},
   JOURNAL = {Math. Nachr.},
  FJOURNAL = {Mathematische Nachrichten},
    VOLUME = {251},
      YEAR = {2003},
     PAGES = {3--18}
}

@article {AmmariBao98,
    AUTHOR = {Ammari, Habib and Bao, Gang},
     TITLE = {Analysis of the diffraction from periodic chiral structures},
   JOURNAL = {C. R. Acad. Sci. Paris S\'er. I Math.},
  FJOURNAL = {Comptes Rendus de l'Acad\'emie des Sciences. S\'erie I.
              Math\'ematique},
    VOLUME = {326},
      YEAR = {1998},
    NUMBER = {12},
     PAGES = {1371--1376}
}

@article{GallinetButetMartin15,
  author = {B. Gallinet and J. Butet and O. J. F. Martin},
  title = {Numerical methods for nanophotonics: Standard problems and 
    future challenges},
  journal = {Laser and Photonics Reviews},
  volume = {9},
  year = {2015},
  pages = {577–603}}

@book {TafloveHagness00,
    AUTHOR = {Taflove, Allen and Hagness, Susan C.},
     TITLE = {Computational electrodynamics: the finite-difference
              time-domain method},
   EDITION = {Second},
 PUBLISHER = {Artech House, Inc., Boston, MA},
      YEAR = {2000},
     PAGES = {xxiv+852}
}

@book {Rumpf22,
  author = {Raymond Rumpf},
  title = {Electromagnetic and Photonic Simulation for the Beginner: 
    Finite-Difference Frequency-Domain in MATLAB},
  publisher = {Artech House, Inc., Boston, MA},
  year = {2020},
  pages = {350} }

@book {Jin02,
    AUTHOR = {Jin, Jianming},
     TITLE = {The finite element method in electromagnetics},
   EDITION = {Second},
 PUBLISHER = {Wiley-Interscience [John Wiley \& Sons], New York},
      YEAR = {2002},
     PAGES = {xxvi+753},
}

@article{BuschKonigNiegmann11,
  author = {Busch, K. and K\"onig, M. and Niegemann, J.},
  title = {Discontinuous {G}alerkin methods in nanophotonics},
  journal = {Laser \& Photonics Reviews},
  volume = {5},
  number = {6},
  pages = {773-809},
  year = {2011}
}

@article{MoharamGaylord81,
author = {M. G. Moharam and T. K. Gaylord},
journal = {J. Opt. Soc. Am.},
number = {7},
pages = {811--818},
publisher = {Optica Publishing Group},
title = {Rigorous coupled-wave analysis of planar-grating diffraction},
volume = {71},
month = {Jul},
year = {1981}
}

@article{MoharamPommetGrannGaylord95,
author = {M. G. Moharam and Drew A. Pommet and Eric B. Grann and 
  T. K. Gaylord},
journal = {J. Opt. Soc. Am. A},
number = {5},
pages = {1077--1086},
publisher = {Optica Publishing Group},
title = {Stable implementation of the rigorous coupled-wave analysis for 
  surface-relief gratings: enhanced transmittance matrix approach},
volume = {12},
month = {May},
year = {1995}
}

@article{LalanneMorris96,
author = {Philippe Lalanne and G. Michael Morris},
journal = {J. Opt. Soc. Am. A},
number = {4},
pages = {779--784},
publisher = {Optica Publishing Group},
title = {Highly improved convergence of the coupled-wave method for 
  {TM} polarization},
volume = {13},
month = {Apr},
year = {1996}
}

@book {KimParkLee12,
    AUTHOR = {Kim, Hwi and Park, Junghyun and Lee, Byoungho},
     TITLE = {Fourier modal method and its applications in computational
              nanophotonics},
 PUBLISHER = {CRC Press, Boca Raton, FL},
      YEAR = {2012},
     PAGES = {xii+313}
}

@article{MartinPiller98,
  title = {Electromagnetic scattering in polarizable backgrounds},
  author = {Martin, Olivier J. F. and Piller, Nicolas B.},
  journal = {Phys. Rev. E},
  volume = {58},
  issue = {3},
  pages = {3909--3915},
  year = {1998},
  month = {Sep},
  publisher = {American Physical Society}
}

@article{DraineFlatau94,
  author = {Bruce T. Draine and Piotr J. Flatau},
  journal = {J. Opt. Soc. Am. A},
  number = {4},
  pages = {1491--1499},
  publisher = {Optica Publishing Group},
  title = {Discrete-Dipole Approximation For Scattering Calculations},
  volume = {11},
  month = {Apr},
  year = {1994}
}

@article{MRBJA93,
  title = {Accurate theoretical analysis of photonic band-gap materials},
  author = {Meade, R. D. and Rappe, A. M. and Brommer, K. D. and 
    Joannopoulos, J. D. and Alerhand, O. L.},
  journal = {Phys. Rev. B},
  volume = {48},
  issue = {11},
  pages = {8434--8437},
  year = {1993},
  month = {Sep}
}

@article{JohnsonJoannopoulos01,
  author = {Steven G. Johnson and J. D. Joannopoulos},
  journal = {Opt. Express},
  number = {3},
  pages = {173--190},
  publisher = {Optica Publishing Group},
  title = {Block-iterative frequency-domain methods for {M}axwell's 
    equations in a planewave basis},
  volume = {8},
  month = {Jan},
  year = {2001}
}

@book {Kress14,
    AUTHOR = {Kress, Rainer},
     TITLE = {Linear integral equations},
   EDITION = {Third},
 PUBLISHER = {Springer-Verlag},
   ADDRESS = {New York},
      YEAR = {2014},
     PAGES = {xvi+412}
}

@book{Lakhtakia94,
  author = {Lakhtakia, A},
  title = {Beltrami Fields in Chiral Media},
  publisher = {WORLD SCIENTIFIC},
  year = {1994}
}

@book {LVV89Book,
  AUTHOR = {Lakhtakia, A. and Varadan, V. K. and Varadan, V. V.},
  TITLE = {Time-harmonic electromagnetic fields in chiral media},
  SERIES = {Lecture Notes in Physics},
  VOLUME = {335},
  PUBLISHER = {Springer-Verlag, Berlin},
  YEAR = {1989},
  PAGES = {viii+121}
}

@book{ChiralityChemistry2024,
  editor = {Janine Cossy},
  booktitle = {Comprehensive Chirality (Second Edition)},
  title = {Comprehensive Chirality (Second Edition)},
  publisher = {Academic Press},
  edition = {Second},
  address = {Oxford},
  pages = {6300},
  year = {2024} }

@book{ChiralityBiology2019,
  author = {Gyula Palyi},
  booktitle = {Biological Chirality},
  title = {Biological Chirality},
  publisher = {Academic Press},
  edition = {First},
  pages = {268},
  year = {2019} }

@book{ChiralityPhysics2023,
  author = {Babaev, Egor and Kharzeev, Dmitri and Larsson, Mats and 
    Molochkov, Alexander and Zhaunerchyk, Vitali},
  title = {Chiral Matter},
  publisher = {WORLD SCIENTIFIC},
  year = {2023},
  doi = {10.1142/13107} }

@article{Tokunaga18,
  author = {Tokunaga, Etsuko and Yamamoto, Takeshi and 
    Ito, Emi and Shibata, Norio},
  date = {2018/11/20},
  doi = {10.1038/s41598-018-35457-6},
  journal = {Scientific Reports},
  number = {1},
  pages = {17131},
  title = {Understanding the Thalidomide Chirality in Biological 
    Processes by the Self-disproportionation of Enantiomers},
  url = {https://doi.org/10.1038/s41598-018-35457-6},
  volume = {8},
  year = {2018} }

@article{Blaschke79,
  author = {Blaschke, G. and Kraft, H.P. and Fickentscher, K. and
    K\"{o}hler, F.},
  title = {Chromatographische Racemattrennung von Thalidomid und 
    teratogene Wirkung der Enantiomere [Chromatographic separation of 
    racemic thalidomide and teratogenic activity of its enantiomers
    (author's transl)]},
  journal = {Arzneimittelforschung},
  year = {1979},
  volume = {29},
  number = {10},
  pages = {1640-2},
  language = {German} }

@book{GroganWinston23,
  author = {Grogan DP, Winston NR},
  title = {Thalidomide},
  howpublished = {\url{https://www.ncbi.nlm.nih.gov/books/NBK557706/}},
  publisher = {StatPearls Publishing, Treasure Island, FL},
  year = {2023}
  }

@article{Mun20,
  author = {Mun, Jungho and Kim, Minkyung and Yang, Younghwan and
    Badloe, Trevon and Ni, Jincheng and Chen, Yang and Qiu, Cheng-Wei and
    Rho, Junsuk},
  year = {2020},
  title = {Electromagnetic chirality: from fundamentals to 
    nontraditional chiroptical phenomena},
  journal = {Light: Science \& Applications},
  pages = {139},
  volume = {9},
  issue = {1}
}

@article{Khaliq23,
author = {Khaliq, Hafiz Saad and Nauman, Asad and Lee, Jae-Won and 
  Kim, Hak-Rin},
title = {Recent Progress on Plasmonic and Dielectric Chiral Metasurfaces: 
  Fundamentals, Design Strategies, and Implementation},
journal = {Advanced Optical Materials},
volume = {11},
number = {16},
pages = {2300644},
year = {2023}
}

@article{Wang25,
author = {Wang, Shenli and Li, Haoyu and Fan, Shengshi and 
  Chen, Lishui and Zhou, Haibo and Pérez-Juste, Jorge and 
  Pastoriza-Santos, Isabel and Wong, Kwok-yin and Zheng, Guangchao},
title = {Chiral Plasmonic Sensors: Fundamentals and Emerging Applications},
journal = {Angewandte Chemie International Edition},
volume = {64},
number = {51},
pages = {e202514816},
doi = {https://doi.org/10.1002/anie.202514816},
url = {https://onlinelibrary.wiley.com/doi/abs/10.1002/anie.202514816},
year = {2025}
}

@article{Li25,
author = {Haoyu Li and Weixiang Ye and Lakshminarayana Polavarapu and 
  Juan Xie and Kwok-Yin Wong and Guangchao Zheng},
journal = {Opt. Lett.},
number = {9},
pages = {2900--2903},
publisher = {Optica Publishing Group},
title = {Chiral plasmonic superlattice resonance based on metasurfaces 
  for chiral molecular sensors},
volume = {50},
month = {May},
year = {2025},
url = {https://opg.optica.org/ol/abstract.cfm?URI=ol-50-9-2900},
doi = {10.1364/OL.559690}
}

@article{Jeong16,
author = {Jeong, Hyeon-Ho and Mark, Andrew G. and 
  Alarc{\'o}n-Correa, Mariana and Kim, Insook and Oswald, Peter and 
  Lee, Tung-Chun and Fischer, Peer},
  date = {2016/04/19},
  doi = {10.1038/ncomms11331},
  journal = {Nature Communications},
  number = {1},
  pages = {11331},
  title = {Dispersion and shape engineered plasmonic nanosensors},
  volume = {7},
  year = {2016}
}

@article{Luo25,
author = {Luo, Taotao and Li, Haoyu and Zhang, Zhicheng and 
  Wang, Shenli and Pan, Xiaobin and Mourdikoudis, Stefanos and 
  Xue, Chao and Li, Junjun and Wong, Kwok-Yin and Zheng, Guangchao},
title = {Super-Heterostructures of Twisted {P}d Nanoarrays Epitaxially 
  Grown on Chiral {A}u Nanorods Boost Circularly Polarized Photocatalysis},
journal = {Advanced Science},
volume = {12},
number = {26},
pages = {2502848},
doi = {https://doi.org/10.1002/advs.202502848},
year = {2025}
}

@article{Kumar18,
author = {Jatish Kumar  and Hasier Eraña  and Elena López-Martínez  and 
  Nathalie Claes  and Víctor F. Martín  and Diego M. Solís  and 
  Sara Bals  and Aitziber L. Cortajarena  and Joaquín Castilla  and 
  Luis M. Liz-Marzán },
title = {Detection of amyloid fibrils in {P}arkinson’s disease using 
  plasmonic chirality},
journal = {Proceedings of the National Academy of Sciences},
volume = {115},
number = {13},
pages = {3225-3230},
year = {2018},
doi = {10.1073/pnas.1721690115},
}

@article{Kim22,
  author = {Kim, Ryeong Myeong and Huh, Ji-Hyeok and Yoo, SeokJae and 
    Kim, Tae Gyun and Kim, Changwon and Kim, Hyeohn and
    Han, Jeong Hyun and Cho, Nam Heon and Lim, Yae-Chan and 
    Im, Sang Won and Im, EunJi and Jeong, Jae Ryeol and 
    Lee, Min Hyung and Yoon, Tae-Young and Lee, Ho-Young and 
    Park, Q-Han and Lee, Seungwoo and Nam, Ki Tae},
  date = {2022/12/01},
  doi = {10.1038/s41586-022-05353-1},
  isbn = {1476-4687},
  journal = {Nature},
  number = {7940},
  pages = {470--476},
  title = {Enantioselective sensing by collective circular dichroism},
  url = {https://doi.org/10.1038/s41586-022-05353-1},
  volume = {612},
  year = {2022}
}

@article{Stratis99,
  title={Electromagnetic scattering problems in chiral media: a review},
  author={Stratis, Ioannis G},
  journal={Electromagnetics},
  volume={19},
  number={6},
  pages={547--562},
  year={1999},
  publisher={Taylor \& Francis}}

@article{AthanasiadisCostakisStratis00,
  author={Athanasiadis, Christodoulos and Costakis, George and 
    Stratis, Ioannis G},
  title={Electromagnetic scattering by a homogeneous chiral obstacle 
    in a chiral environment},
  journal={IMA journal of applied mathematics},
  volume={64},
  number={3},
  pages={245--258},
  year={2000},
  publisher={Oxford University Press}}

@article{AthanasiadisMartinStratis99,
  author={Athanasiadis, C and Martin, Paul A and Stratis, Ioannis G},
  title={Electromagnetic scattering by a homogeneous chiral obstacle: 
    boundary integral equations and low-chirality approximations},
  journal={SIAM Journal on Applied Mathematics},
  volume={59},
  number={5},
  pages={1745--1762},
  year={1999},
  publisher={SIAM}}

@article{Rojas94,
  title={Integral equations for {EM} scattering by homogeneous/inhomogeneous 
    two-dimensional chiral bodies},
  author={Rojas, RG},
  journal={IEEE Proceedings-Microwaves, Antennas and Propagation},
  volume={141},
  number={5},
  pages={385--392},
  year={1994},
  publisher={IET}}

@article{GuoWang22,
  author = {Jun Guo and Haibing Wang},
  title = {On the direct and inverse electromagnetic scattering by 
    chiral media},
  journal = {Journal of Differential Equations},
  volume = {317},
  pages = {495-523},
  year = {2022}}

@article{FengWangZhang21,
  author={Feng, Lixin and Wang, Haibing and Zhang, Lei},
  title={The forward and inverse problems for the scattering of 
    obliquely incident electromagnetic waves in a chiral medium},
  journal={Journal of Differential Equations},
  volume={284},
  pages={102--125},
  year={2021},
  publisher={Elsevier}}

@article{Nguyen16,
  author={Nguyen, Dinh-Liem},
  title={The Factorization Method for the {D}rude--{B}orn--{F}edorov Model
    for Periodic Chiral Structures},
  journal={Inverse Problems \& Imaging},
  volume={10},
  number={2},
  year={2016}}

@article{PotthastStratis03,
  author={Potthast, Roland and Stratis, Ioannis G},
  title={On the domain derivative for scattering by impenetrable 
    obstacles in chiral media},
  journal={IMA journal of applied mathematics},
  volume={68},
  number={6},
  pages={621--635},
  year={2003},
  publisher={OUP}}

@article{Gerlach99,
  title={The two-dimensional electromagnetic inverse scattering problem 
    for chiral media},
  author={Gerlach, Thomas},
  journal={Inverse Problems},
  volume={15},
  number={6},
  pages={1663--1675},
  year={1999}}

@article{deMonvelBoutetShepelsky97,
  author={de Monvel, Anne Boutet and Shepelsky, Dimitri},
  title={Direct and inverse scattering problem for a stratified 
    nonreciprocal chiral medium},
  journal={Inverse Problems},
  volume={13},
  number={2},
  pages={239--251},
  year={1997}}

@article{ZhangMa05,
  author = {Deyue Zhang and Fuming Ma},
  title = {An inverse electromagnetic scattering problem for periodic 
    chiral structures},
  journal = {Journal of Physics: Conference Series},
  year = {2005},
  month = {jan},
  volume = {12},
  number = {1},
  pages = {180}}

@article{ZhangMa07,
 author = {Deyue Zhang and Fuming Ma},
 title = {A Finite Element Method with Perfectly Matched Absorbing
   Layers for the Wave Scattering by a Periodic Chiral Structure},
 journal = {Journal of Computational Mathematics},
 number = {4},
 pages = {458--472},
 publisher = {Institute of Computational Mathematics and 
   Scientific/Engineering Computing},
 volume = {25},
 year = {2007}}

@article{ZhangGuoGongWang12,
  author = {Zhang, Deyue and Guo, Yukun and Gong, Chengchun and Wang, Guan},
  title = {Numerical analysis for the scattering by obstacles in a 
    homogeneous chiral environment},
  date = {2012/01/01},
  journal = {Advances in Computational Mathematics},
  number = {1},
  pages = {3--20},
  volume = {36},
  year = {2012}}

@article{Homola08,
  author = {Homola, J.},
  title = {Surface Plasmon Resonance Sensors for Detection of 
    Chemical and Biological Species},
  year = {2008},
  journal = {Chemical Reviews},
  volume = {108},
  number = {2},
  pages = {462--493}}

@book {BaoLi22,
    AUTHOR = {Bao, Gang and Li, Peijun},
     TITLE = {Maxwell's equations in periodic structures},
    SERIES = {Applied Mathematical Sciences},
    VOLUME = {208},
 PUBLISHER = {Springer, Singapore; Science Press Beijing, Beijing},
      YEAR = {2022},
     PAGES = {xi+355}
}

@book{Raether88,
  fauthor = {Raether, Heinz},
  author = {Raether, H.},
  title = {Surface plasmons on smooth and rough surfaces and on gratings},
  publisher = {Springer},
  address = {Berlin},
  year = {1988}
  }

@book{Maier07,
  fauthor = {Maier, Stefan A.},
  author = {Maier, S. A.},
  title = {Plasmonics: Fundamentals and Applications},
  publisher = {Springer},
  address = {New York},
  year = {2007},
  pages = {xxv+223}
}

@book{EnochBonod12,
  fauthor = {Enoch, Stefan and Bonod, Nicolas},
  author = {Enoch, S. and Bonod, N.},
  title = {Plasmonics: From Basics to Advanced Topics},
  series = {Springer Series in Optical Sciences},
  publisher = {Springer},
  address = {New York},
  year = {2012},
  pages = {xvi+321}
}

@book {Jackson75,
    FAUTHOR = {Jackson, John David},
    AUTHOR = {Jackson, J. D.},
     TITLE = {Classical electrodynamics},
   EDITION = {Second},
 PUBLISHER = {John Wiley \& Sons Inc.},
   ADDRESS = {New York},
      YEAR = {1975},
     PAGES = {xxii+848}
}

@book{Yeh05,
  title={Optical waves in layered media},
  author={Yeh, Pochi},
  volume={61},
  year={2005},
  publisher={Wiley-Interscience}
}

@book {TKS85,
   AUTHOR = {Tsang, L. and Kong, J. A. and Shin, R. T.},
    TITLE = {Theory of Microwave Remote Sensing},
PUBLISHER = {Wiley},
  ADDRESS = {New York},
     YEAR = {1985}
}

@book {S02,
  author = {Shull, Peter J.},
  title = {Nondestructive Evaluation: Theory, Techniques, and Applications},
  publisher = {Marcel Dekker},
  year = {2002}
}

@book {Evans10,
    AUTHOR = {Evans, Lawrence C.},
     TITLE = {Partial differential equations},
 PUBLISHER = {American Mathematical Society},
   EDITION = {Second},
   ADDRESS = {Providence, RI},
      YEAR = {2010},
     PAGES = {xxiii+755}
}

@book {GottliebOrszag77,
   FAUTHOR = {Gottlieb, David and Orszag, Steven A.},
   AUTHOR = {Gottlieb, D. and Orszag, S. A.},
    TITLE = {Numerical analysis of spectral methods: theory and
             applications},
     NOTE = {{C}BMS-NSF Regional Conference Series in Applied Mathematics, No.
             26},
PUBLISHER = {Society for Industrial and Applied Mathematics},
  ADDRESS = {Philadelphia, Pa.},
     YEAR = {1977},
    PAGES = {v+172}
}

@book {Boyd01,
    AUTHOR = {Boyd, John P.},
     TITLE = {Chebyshev and {F}ourier spectral methods},
   EDITION = {Second},
 PUBLISHER = {Dover Publications Inc.},
   ADDRESS = {Mineola, NY},
      YEAR = {2001},
     PAGES = {xvi+668}
}

@book {BakerGravesMorris96,
   AUTHOR = {Baker, Jr., George A. and Graves-Morris, Peter},
    TITLE = {{P}ad\'e approximants},
  EDITION = {Second},
PUBLISHER = {Cambridge University Press},
  ADDRESS = {Cambridge},
     YEAR = {1996},
    PAGES = {xiv+746}
}

@book {BenderOrszag78,
    AUTHOR = {Bender, Carl M. and Orszag, Steven A.},
     TITLE = {Advanced mathematical methods for scientists and engineers},
      NOTE = {International Series in Pure and Applied Mathematics},
 PUBLISHER = {McGraw-Hill Book Co.},
   ADDRESS = {New York},
      YEAR = {1978},
     PAGES = {xiv+593}
}

\end{document}